\documentclass[12pt]{amsart}
\usepackage[margin=1.2in]{geometry}

\usepackage{amsmath,amssymb,amsfonts,amsthm,mathtools}
\usepackage{enumitem}
\usepackage{hyperref}
\usepackage{verbatim}
\usepackage{xcolor}

\hypersetup{
  colorlinks=true,
  linkcolor=blue,
  citecolor=blue,
  urlcolor=blue
}

\theoremstyle{plain}
\newtheorem{theorem}{Theorem}[section]
\newtheorem{lemma}[theorem]{Lemma}
\newtheorem*{lemma*}{Lemma} 
\newtheorem{proposition}[theorem]{Proposition}
\newtheorem{corollary}[theorem]{Corollary}

\theoremstyle{definition}

\theoremstyle{remark}
\newtheorem{remark}[theorem]{Remark}

\newcommand{\E}{\mathbb{E}}

\newcommand{\R}{\mathbb{R}}

\newcommand{\T}{\mathbb{T}}
\newcommand{\Z}{\mathbb{Z}}
\newcommand{\C}{\mathbb C}

\newcommand{\ep}{\varepsilon}

\newcommand{\Lip}{\mathrm{Lip}}

\newcommand{\essinf}{\operatorname*{ess\,inf}}

\newcommand{\ol}{\overline}

\newcommand{\tr}{{\rm tr}\,}

\begin{document}

\title[Homogenization of viscous quadratic Hamilton-Jacobi equations]{Sharp global and almost everywhere convergence rates for periodic homogenization of viscous quadratic Hamilton-Jacobi equations}

\author[Z. LIU, H. V. TRAN, Y. YU]
{Ziran Liu, Hung V. Tran, Yifeng Yu}

\thanks{
H. V. Tran is partially supported by NSF grant DMS-2348305. 
}

\address[Z. Liu]
{Shanghai Institute for Mathematics and Interdisciplinary Sciences (SIMIS), Shanghai, China, 200433}
\address[Z. Liu]
{Research Institute of Intelligent Complex Systems, Fudan University, Shanghai 200433, China}
\email{zliu@simis.cn}

\address[H. V. Tran]
{
Department of Mathematics, 
University of Wisconsin-Madison, Van Vleck Hall, 480 Lincoln Drive, Madison, Wisconsin 53706, USA}
\email{hung@math.wisc.edu}

\address[Y. Yu]
{
Department of Mathematics, 
University of California at Irvine, 
California 92697, USA}
\email{yifengy@uci.edu}


\begin{abstract}
We study the periodic homogenization of the viscous Hamilton--Jacobi equation
\begin{equation*}
u_t^\ep + {1\over 2}|Du^\ep|^2+V\left({x\over \ep}\right)={\ep\over 2} \Delta u^\ep
\qquad \text{in }\R^n\times(0,\infty),
\end{equation*}
with initial datum $g\in W^{1,\infty}(\mathbb{R}^n)$, where $V$ is Lipschitz continuous and $\mathbb{Z}^n$-periodic.
We prove the sharp global estimate
\[
|u^\ep(x,t)-u(x,t)|\leq \ep\left(C+\frac{n}{2}\log \left(\frac{\max\{t,\ep\}}{\ep}\right)\right) \quad \text{for all $(x,t)\in \R^n\times [0,\infty)$},
\]
where $\ep\in (0,1]$, $u$ solves the limiting (homogenized) equation and $C>0$ is a constant depending only on $\|Dg\|_{L^{\infty}(\R^n)}$, $\|DV\|_{L^\infty(\R^n)}$, and $n$.

We further show that if $g$ is locally semiconcave, then
\begin{equation*}
|u^\ep(x,t)-u(x,t)|\leq C_{x,t}\ep \qquad \text{ for a.e. $(x,t)\in \R^n\times (0,\infty)$},
\end{equation*}
where $C_{x,t}$ depends on $(x,t)$, $\|Dg\|_{L^{\infty}(\R^n)}$, and $\|DV\|_{L^\infty(\R^n)}$. 
More precisely, the above improved rate holds at every point $(x,t)$ where $u(\cdot,t)$ is twice differentiable at $x$.
In particular, this occurs for a.e. $x\in \R^n$, since $u(\cdot,t)$ is locally semiconcave.

We conclude by raising the open problem of whether the same $O(\ep |\log\ep|)$ rate remains valid for general strictly convex Hamiltonians or general periodic diffusions.
\end{abstract}

\maketitle
\section{Introduction}
For every $\ep>0$, we consider the Cauchy problem
\begin{equation}\label{eq:C-ep}
\begin{cases}
u_t^\ep + H\left(\frac{x}{\ep},Du^\ep\right)={\ep\over 2} \Delta u^\ep
\qquad &\text{in }\R^n\times(0,\infty),\\
u^\ep(x,0)=g(x)
\qquad &\text{on }\R^n,
\end{cases}
\end{equation}
and denote by $u^\ep \in C(\R^n\times[0,\infty))$ its viscosity solution. The Hamiltonian
$H=H(y,p):\R^n\times\R^n\to\R$ and the initial data $g:\R^n\to\R$
are prescribed. Assume that
\begin{itemize}
\item[(A1)] $H\in \Lip_{\rm loc}(\R^n\times\R^n)$, and for each $p\in\R^n$, the map
$y\mapsto H(y,p)$ is $\Z^n$-periodic;

\item[(A2)] 
\[
\lim_{|p|\to\infty}\essinf_{y\in\R^n}
\left(|H(y,p)|^2-(n+1)D_yH(y,p)\cdot p\right)=+\infty;
\]

\item[(A3)] $g\in W^{1,\infty}(\mathbb{R}^n)$.
\end{itemize}
A basic example of $H$ is the quadratic case (mechanical Hamiltonian): 
\[H(y,p)={1\over 2}|p|^2+V(y) \qquad \text{ for } (y,p)\in \R^n\times \R^n.
\]

Under assumptions \textnormal{(A1)}--\textnormal{(A3)}, it is known that, as
$\ep\to0+$, the family $\{u^\ep\}$ converges locally uniformly in
$\R^n\times[0,\infty)$ to a function $u\in C(\R^n\times[0,\infty))$.
Moreover, $u$ is the viscosity solution of the homogenized equation
\begin{equation}\label{eq:C}
\begin{cases}
u_t+\ol{H}(Du)=0
\qquad &\text{in }\R^n\times(0,\infty),\\
u(x,0)=g(x)
\qquad &\text{on }\R^n.
\end{cases}
\end{equation}
See \cite{Ev1}. Here $\ol{H}\in C(\R^n)$ denotes the effective
Hamiltonian associated with $H$, defined through the corresponding cell
problem: for any $p\in \R^n$, there exists a unique number $\ol H(p)\in \R$ such that the following cell problem has a $\Z^n$-periodic viscosity solution
\[
-\frac{1}{2}\Delta v+H(y,p+Dv)=\ol H(p) \qquad \text{ in $\R^n$}.
\]
If $H$ is convex in $p$, then $\overline H$ is also convex in $p$.
In this case, we denote by $\overline L$ the effective Lagrangian associated with $\overline H$:
\begin{equation}\label{eq:eff-lagrangian}
\overline L(q)=\sup_{p\in \R^n}\left\{q\cdot p-\overline H(p)\right\}.
\end{equation}

A natural question is to identify the convergence rate of $u^\ep\to u$ as $\ep\to 0$, which belongs to the very active area of quantitative homogenizations.  For general $H$ satisfying \textnormal{(A1)}--\textnormal{(A3)}, it was proved in \cite{QSTY} that
\[
|u^\ep-u|\leq O(\sqrt{\ep}). 
\]
See \cite{CDI,CCM} for the earlier $O(\ep^{1/3})$-rate.
The argument is based on the perturbed test function method in \cite{Ev1}, together with a careful choice of scaling in $\ep$. The above rate turns out to be optimal in the sense that there exist particular choices of $H$ and $g$ satisfying assumptions \textnormal{(A1)}--\textnormal{(A3)}, where $H$ is not strictly convex in $p$, such that the convergence rate is $\sqrt{\ep}$.

In this paper, we consider the most interesting example $H(y,p)={1\over 2}|p|^2+V(y)$. The equation \eqref{eq:C-ep} becomes
\begin{equation}\label{eq:QC}
\begin{cases}
u_t^\ep + {1\over 2}|Du^\ep|^2+V\left({x\over \ep}\right)={\ep\over 2} \Delta u^\ep
\qquad &\text{in }\R^n\times(0,\infty),\\
u^\ep(x,0)=g(x)
\qquad &\text{on }\R^n.
\end{cases}
\end{equation}
Numerical results in \cite{QSTY} suggest that the convergence rate is faster, approaching $\ep$. In this paper, we will confirm this observation. A new approach is needed since the method used in \cite{QSTY} is too coarse to obtain the optimal convergence rate for the quadratic case.

Below are our main results. 

\begin{theorem}\label{theo:main} 
Assume that $H(y,p)={1\over 2}|p|^2+V(y)$ for $(y,p)\in \R^n\times \R^n$, where $V\in \Lip(\R^n)$ is $\Z^n$-periodic.
Assume further that $g\in W^{1,\infty}(\mathbb{R}^n)$.
Then, there exists $C=C(\|Dg\|_{L^{\infty}(\R^n)},\|DV\|_{L^\infty(\R^n)},n)>0$ such that, for $(x,t)\in \R^n\times [0,\infty)$ and $\ep\in (0,1)$,
\begin{equation}\label{eq:main-rate}
|u^\ep(x,t)-u(x,t)|\leq \ep\left(C+\frac{n}{2}\log \left(\frac{\max\{t,\ep\}}{\ep}\right)\right).
\end{equation}
\end{theorem}

\begin{theorem}\label{theo:main2} 
Assume that $H(y,p)={1\over 2}|p|^2+V(y)$ for $(y,p)\in \R^n\times \R^n$, where $V\in \Lip(\R^n)$ is $\Z^n$-periodic.
Assume further that $g\in W^{1,\infty}(\mathbb{R}^n)$ is locally semiconcave.
Then, for each $t>0$, $u(\cdot,t)$ is locally semiconcave and is twice differentiable almost everywhere in $x\in \R^n$.

For $(x,t)\in \R^n\times (0,\infty)$, if $u$ is twice differentiable at $x$, then
\begin{equation}\label{eq:main-rate-ep}
|u^\ep(x,t)-u(x,t)|\leq C_{x,t}\ep
\end{equation}
for a constant $C_{x,t}$ depending on the point $(x,t)$, $\|Dg\|_{L^{\infty}(\R^n)}$, and $\|DV\|_{L^\infty(\R^n)}$.
\end{theorem}
It is worth to note that Lemma \ref{lem:Q-Nondeneracy} in the proof of Theorem \ref{theo:main2} was supplied by ChatGPT, and verified carefully by the authors.

\begin{remark}
    When the diffusion term is absent, the optimal convergence rate is $O(\ep)$ for all convex coercive Hamiltonians $H$; see \cite{TY}. 
    See \cite{HJ23, HJMT25, MN26, MNT25, HT25} and the references therein for optimal rates of first-order problems in various settings.
    
    In the viscous case, the $\ep|\log \ep|$ rate is sharp.  See \cite[Proposition 4.4]{QSTY} for the matching lower-bound direction in the special case where $V$ is constant (the vanishing viscosity case). 
    
It remains an interesting open problem whether the same $\ep |\log \ep|$ bound holds for other strictly convex Hamiltonians, for example
\[
H(p)=\sqrt{1+|p|^2},
\]
which is related to the $G$-equation with curvature effects in a shear ambient flow; see \cite{XYR} and also \cite{LXY}.
Likewise, one may ask what the optimal convergence rate is when the term $\ep \Delta u^\ep$ is replaced by a more general diffusion operator of the form $\ep \tr(A(\frac{x}{\ep}) D^2 u^\ep)$, including the degenerate setting.

We refer the reader to \cite{CG2025, CD2025, WZ2025} and the references therein for interesting results on the $O(\ep|\log \ep|)$ convergence of the vanishing viscosity process for general strictly convex Hamiltonians.
\end{remark}

\begin{remark}\label{probrmk1}

Let us see Theorem~\ref{theo:main} through a probabilistic (operator) semigroups point of view.

\begin{itemize}
    \item [1.]For the mechanical Hamiltonian
\[
H(y,p)=\frac12 |p|^2+V(y),
\]
the standard Hopf--Cole transform converts the viscous Hamilton--Jacobi equation into the linear parabolic equation
\[
w_t-\frac12\Delta w - V(x)w=0.
\]
Equivalently, if one sets
\[
\mathcal L:=\frac12\Delta +V,
\]
then the transformed solution is governed by the (operator) semigroup $e^{t\mathcal L}$:
\[
w(\cdot,t)=e^{t\mathcal L}w(\cdot,0).
\]
\item [2.]For each $p\in\R^n$, the effective Hamiltonian $\overline H(p)$ is not only the quantity in the cell problem of periodic homogenization, but also the principal eigenvalue of the tilted periodic operator
\[
\mathcal L_p \phi
=
\frac12 \Delta \phi - p\cdot D\phi
+\left(\frac12 |p|^2+V(y)\right)\phi
\qquad \text{in }\T^n.
\]
Let $r_p(x)=e^{-v_p(x)}$. Then
\[
{\mathcal L_p r_p=\overline H(p)\,r_p.}
\]
Equivalently, if $h_p(x)=e^{-p\cdot x-v_p(x)}$, then
\[
{\mathcal L h_p=\overline H(p)\,h_p.}
\]
The corresponding ground-state, or Doob $h$-transform, produces the conservative diffusion with generator
\[
\mathcal G_p f=\frac12\Delta f+b_p(x)\cdot Df,
\qquad {}
b_p(x)=-p-Dv_p(x),
\]
{whose transition density with respect to Lebesgue measure is denoted by $\widetilde p_p(t,x,y)$. Accordingly,}
\[
{K(t,x,y)=e^{t\overline H(p)}\frac{h_p(x)}{h_p(y)}\widetilde p_p(t,x,y).}
\]
\item [3.] By the Feynman--Kac formula,
\[
(e^{t\mathcal L}f)(x)
=
\mathbb E_x\!\left[
\exp\!\left(\int_0^t V(X_s)\,ds\right)f(X_t)
\right],
\]
so that the kernel $K(t,x,y)$ may be viewed as the kernel of a periodic Feynman--Kac weighted diffusion.

\end{itemize}

Following this probabilistic point of view, one can see that the proof naturally separates into two parts: the exponential scale is encoded by the tilted principal eigenvalue $\overline H(p)$, while the sharp pre-exponential information is governed by the long-time ballistic asymptotics of the transformed diffusion kernel $\widetilde p_p(t,x,y)$. 

In such a sense, the argument connects three complementary viewpoints on the same problem: quantitative homogenization for viscous Hamilton--Jacobi equations, semigroup or Feynman--Kac methods and Doob $h$-transforms, and the spectral description of long-time Gaussian asymptotics in periodic media; see, for example, \cite{ConcaVanninathan1997,KaratzasShreve,KotaniSunada2000,Kuchment1993,Norris1997,Pinsky,SimonFIQP, HKN20}.

\end{remark}

\subsection*{Organization of the paper}
In Section \ref{sec:prelim}, we review basic properties of the effective Hamiltonian and effective Lagrangian. The proofs of Theorems \ref{theo:main} and \ref{theo:main2} are given in Section \ref{sec:proof of main thm}. A key ingredient is the so-called ballistic estimate" (Proposition \ref{thm:ballistic-compact}), which is proved in Section \ref{sec:ballistic} as a corollary of \cite[Theorem 1.3]{Norris1997} via the Doob $h$-transform. Finally, in Appendix \ref{sec:appen}, for the reader’s convenience, we provide a self-contained proof of the large-time Gaussian asymptotics for the heat kernel and derive a sharp ballistic asymptotic formula for the original Schr\"odinger kernel. Although these more delicate estimates are not needed for our main purpose, both the conclusions and the arguments may be useful in other contexts.

\subsection*{Acknowledgment}
Lemma \ref{lem:Q-Nondeneracy} in the proof of Theorem \ref{theo:main2} was supplied by ChatGPT, and verified carefully by the authors.

\section{Preliminaries} \label{sec:prelim}
In this section, we review some basic properties of the effective Hamiltonian $\overline H$. For $p\in \R^n$, let $v_p$ be the unique $\Z^n$-periodic solution to 
\begin{equation}\label{eq:Qcell}
-{1\over 2}\Delta v_p+{1\over 2}|p+Dv_p|^2+V(x)=\overline H(p)  \qquad \text{in $\R^n$}
\end{equation}
subject to $v_p(0)=0$.  

Let $\alpha\in (0,1)$, $\T^n=\R^n/\Z^n$, $X=\left\{\, v\in C^{2,\alpha}(\mathbb{T}^n)\;:\;v(0)=0 \,\right\}\times \mathbb{R}$,
and $Y=C^{0,\alpha}(\mathbb{T}^n)$.
Apply the Banach implicit function theorem to the map
\begin{align*}
    \Theta:\R^n\times X&\to Y\\
\Theta(p,(v,\lambda))&=-{1\over 2}\Delta v+{1\over 2}|p+Dv|^2+V-\lambda,
\end{align*}
we have that 
\begin{equation}\label{eq:smooth-on-p}
p\mapsto \left(v_p, \overline H(p)\right)\qquad \text{ is } C^{\infty}. 
\end{equation}

For the reader's convenience, we prove the following properties of $\overline H(p)$ that are well known to experts. 
\begin{lemma}\label{lem:doob-h-transform} For all $p\in \R^n$,

\begin{enumerate}
\item $\overline H$ is strictly convex in $p$, i.e., $D^2\overline H(p)>0$ for $p\in \R^n$;
    \item ${1\over 2}|p|^2+\min_{\R^n}V\leq \overline H(p)\leq {1\over 2}|p|^2+\max_{\R^n}V$ for all $p\in \R^n$;
    \item  $H$ is an even function, i.e., $\overline H(p)=\overline H(-p)$ for $p\in \R^n$.
\end{enumerate}
\end{lemma}

\begin{proof} 

We prove (1) by establishing a more general statement. 
Consider the cell problem
\[
-{1\over 2}\Delta v_p+H(p+Dv_p)+V(x)=\overline H(p)  \qquad \text{ in $\R^n$}
\]
for a smooth and strictly convex function $H=H(p)$. In our case, $H(p)={1\over 2}|p|^2$. 

In view of \eqref{eq:smooth-on-p}, differentiating the above equation twice with respect to \(p\) in the direction \(\xi\in\mathbb{R}^n\), we obtain
$$
-{1\over 2}\Delta w+q\cdot D^2H(p+Dv_p) q+DH(p+Dv_p)\cdot Dw=\xi\cdot D^2\overline H(p) \xi. 
$$
Here, $w(x)=\xi\cdot D^2_pv_p(x)\xi$ and  $q(x)=(I_n+D^2_pv_p(x)) \xi$. 

Since $q\cdot D^2H(p+Dv_p) q\geq  0$, if $\xi\cdot D^2\overline H(p) \xi \leq 0$, the strong maximum principle and $w(0)=0$ imply that 
\[
w\equiv 0 \qquad \text{ and } \qquad q\cdot D^2H(p+Dv_p)q=\xi\cdot D^2\overline H(p) \xi=0.
\]
So $q=0$. Since $\xi\cdot q=|\xi|^2+w$,  we deduce that $\xi=0$. Hence $D^2\overline H(p)>0$.

\medskip

Claim (2) follows immediately by considering the maximum and minimum points of $v_p$.

\medskip

We now prove (3).  
For $p\in \R^n$, let $r_p=e^{-v_p}$. Then
\[
\begin{cases}
(-{1\over 2}\Delta+p\cdot D-V)r_p=E_pr_p\qquad &\text{ in $\T^n$},\\[3mm]
(-{1\over 2}\Delta-p\cdot D-V)r_{-p}=E_{-p}r_{-p}\qquad &\text{ in $\T^n$}.
    \end{cases}
    \]
Here $E_p={1\over 2}|p|^2-\overline H(p)$. Note that the above two operators are dual.  Then $E_p=E_{-p}$.  This leads to $\overline H(p)=\overline H(-p)$.

\end{proof}

As an immediate corollary, we have that
\begin{corollary}\label{cor:L bar} {Let $\overline L$ be the corresponding effective Lagrangian. Then:}
\begin{enumerate}
    \item {$\overline L\in C^\infty(\R^n)$ is strictly convex ($D^2\overline L>0$) and even in $q$;}
    \item {for every $q\in \R^n$,}
    \[
    {\frac12|q|^2-\max_{\R^n}V\leq \overline L(q)\leq \frac12|q|^2-\min_{\R^n}V.}
    \]
\end{enumerate}

\end{corollary}

Now let 
\begin{equation}\label{eq:invariantmeasure}
\pi_p=Cr_pr_{-p}.
\end{equation}
The constant $C>0$ is chosen so that $\int_{\T^n}\pi_p\,dx=1$. This is the same as the measure $\sigma$ defined in \cite{Ev4}.

\begin{lemma}\label{lem:piinvariant}
We have that
\[
-\frac12\Delta \pi_p-\operatorname{div}\big((p+Dv_p)\pi_p\big)=0
\qquad\text{ in }\mathbb T^n.
\]
\end{lemma}

\begin{proof}
Since $C$ is a normalization constant, it suffices to prove the result for
\[
\pi:=r_pr_{-p}.
\]
Recall that
\[
\left(-\frac12\Delta+p\cdot D-V\right)r_p=E_pr_p,
\qquad
\left(-\frac12\Delta-p\cdot D-V\right)r_{-p}=E_pr_{-p},
\]
with $E_{-p}=E_p$, and
\[
r_p=e^{-v_p},
\qquad
Dr_p=-r_pDv_p.
\]
Hence
\[
p+Dv_p=p-\frac{Dr_p}{r_p},
\]
and therefore
\[
(p+Dv_p)\pi
=
\left(p-\frac{Dr_p}{r_p}\right)r_pr_{-p}
=
pr_pr_{-p}-r_{-p}Dr_p.
\]

Next,
\[
\frac12\Delta(r_pr_{-p})
=
\frac12 r_{-p}\Delta r_p+\frac12 r_p\Delta r_{-p}+Dr_p\cdot Dr_{-p},
\]
while
\[
\operatorname{div}\big(pr_pr_{-p}-r_{-p}Dr_p\big)
=
r_{-p}p\cdot Dr_p+r_pp\cdot Dr_{-p}
-Dr_{-p}\cdot Dr_p-r_{-p}\Delta r_p.
\]
Subtracting the two identities gives
\[
\frac12\Delta\pi+\operatorname{div}\big((p+Dv_p)\pi\big)
=
-r_{-p}\left(\frac12\Delta r_p-p\cdot Dr_p\right)
+
r_p\left(\frac12\Delta r_{-p}+p\cdot Dr_{-p}\right).
\]
By the equations for $r_p$ and $r_{-p}$,
\[
\frac12\Delta r_p-p\cdot Dr_p=-(V+E_p)r_p,
\qquad
\frac12\Delta r_{-p}+p\cdot Dr_{-p}=-(V+E_p)r_{-p}.
\]
Substituting these two identities, we obtain
\[
\frac12\Delta\pi+\operatorname{div}\big((p+Dv_p)\pi\big)=0.
\]
Multiplying by $C$, we conclude that
\[
\frac12\Delta \pi_p+\operatorname{div}\big((p+Dv_p)\pi_p\big)=0
\qquad\text{in }\mathbb T^n.
\]
\end{proof}

\begin{lemma} {For every $p\in \R^n$,}
\[
{D\overline H(p)=\int_{\T^n}(p+Dv_p)\pi_p\,dx.}
\]
\end{lemma}

\begin{proof}
   {For $j\in\{1,\dots,n\}$, let $w_j:=\partial_{p_j}v_p$. Differentiating the cell problem \eqref{eq:Qcell} with respect to $p_j$ gives}
    \[
{-\frac12\Delta w_j+(p+Dv_p)\cdot(e_j+Dw_j)=\partial_{p_j}\overline H(p)
\qquad\text{in }\T^n.}
\]
{Multiplying by $\pi_p$, integrating over $\T^n$, and using Lemma~\ref{lem:piinvariant}, we obtain}
\[
{\partial_{p_j}\overline H(p)=\int_{\T^n}(p_j+\partial_jv_p)\pi_p\,dx.}
\]
{Since this holds for every $j$, the conclusion follows.}
\end{proof}
We refer the reader to \cite{Gomes,Ev4} for further interesting identities involving $\pi_p$ and $v_p$.

\section{Proof of Theorems \ref{theo:main} and \ref{theo:main2}} \label{sec:proof of main thm}

Owing to translation and scaling, we may reduce the proof of \eqref{eq:main-rate} to the case $(x,t)=(0,1)$.  
We first apply the Hopf–Cole transform 
\[
u^{\ep}(x,t)=-\ep \log w\left({x\over \ep}, {t\over \ep}\right).
\]
This converts the equation (\ref{eq:QC}) to a linear equation satisfied by $w$:
$$
\begin{cases}
w_t-{1\over 2}\Delta w-{V}(x)w=0\\[3mm]
w(x,0)=e^{-{g(\ep x)\over \ep}}.
\end{cases}
$$
Therefore
\begin{equation}\label{eq:u-ep-formula}
u^{\ep}(0,1)=-\ep \log\left(\ep^{-n}\int_{\R^n}K\left( {1\over \ep},0, {y\over \ep}\right)e^{-{g(y)\over \ep}}\,dy\right).
\end{equation}
Here $K(t,x,y)$ is the fundamental solution of the Sch\"ordinger operator $\partial_t-\mathcal{L}$ for 
\[
\mathcal{L}={1\over 2}\Delta +V. 
\]

The next step is to provide an accurate estimate of the asymptotic behavior of $K(t,x,y)$. By the Feynman--Kac formula,
\[
K(t,x,y)
=
p_t(x-y)\,
\mathbb{E}^{\,x\to y}_t
\left[
\exp\!\left(\int_0^t V(B_s)\,ds\right)
\right],
\]
where the expectation is over the Brownian bridge from $x$ to $y$ in time $t$. This leads to the following rough estimate of $K$.
\begin{equation}\label{eq:rough-bound-k}
e^{-t\|V\|_{L^\infty}}\, p_t(x-y)
\;\le\;
K(t,x,y)
\;\le\;
e^{\,t\|V\|_{L^\infty}}\, p_t(x-y),
\end{equation}
where
\[
p_t(z)=\frac{1}{(2\pi t)^{n/2}}\,e^{-\frac{|z|^2}{2t}}.
\]

To obtain the $\ep|\log \ep|$ convergence rate, we need the following  {\it ballistic estimate}.

\begin{proposition}[Ballistic kernel bounds on compact velocity sets]\label{thm:ballistic-compact}
Let $V_0\subset \mathbb{R}^n$ be compact. Then there exist constants
$t_0\ge 1$ and $0<C_1\le C_2<\infty$ such that, for every $q\in V_0$ and every $t\ge t_0$,
\begin{equation}\label{eq:ballistic}
\frac{C_1}{{t}^{n/2}}e^{-t\overline{L}(q)}
\le
K(t,0,-qt)
\le
\frac{C_2}{{t}^{n/2}}e^{-t\overline{L}(q)}.
\end{equation}
\end{proposition}

\begin{remark}
For the proof of Theorem~\ref{theo:main}, only the two-sided bounds \eqref{eq:ballistic} are needed. Sharper asymptotic formulas for $K(t,x,y)$ were already available earlier: the one-dimensional case was established in \cite[Theorem~1.1]{T2008}, and higher-dimensional large-deviation asymptotics were announced in \cite{Agmon} and later proved in the more general branching-diffusion setting in \cite[Theorem~2.2]{HKN20}. Proposition~\ref{thm:ballistic-compact} should therefore be viewed only as the weaker corollary tailored to the present argument. 
In Appendix \ref{sec:appen}, we record a sharper asymptotic, Proposition~\ref{prop:A1-sharp-ballistic}, obtained by our Bloch--Floquet/Doob-transform approach.
\end{remark}

\begin{remark}
The Hopf--Cole transform linearizes \eqref{eq:QC} into a parabolic/Schr\"odinger equation with potential $V$ as a technique in PDE.
Probabilistically, the kernel $K(t,x,y)$ turns out to be the kernel of a Feynman--Kac weighted diffusion, as we have explained in Remark \ref{probrmk1}. Thus the ballistic estimate above may be understood as a precise long-time asymptotic for a weighted diffusion kernel. After the Doob $h$-transform given in Subsection \ref{subsec:doob-h}, this weighted kernel becomes the heat kernel of a genuine periodic diffusion.
\end{remark}
  
We postpone the proof of Proposition \ref{thm:ballistic-compact} to Section \ref{sec:ballistic}.

\subsection{Proof of Theorem \ref{theo:main}}
Before presenting the proof of Theorem \ref{theo:main}, we need one final preparatory step.

\begin{lemma}\label{lem:bound-small-time}
There exists $C=C(\|DV\|_{L^\infty},\|Dg\|_{L^\infty},n)>0$ such that, for $(x,t)\in \R^n\times [0,\infty)$,
\begin{equation}\label{eq:small-time-bound}
|u(x,t)-g(x)|\le Ct\qquad \text{ and } \qquad |u^\ep(x,t)-g(x)|\le C(t+\ep).
\end{equation}
\end{lemma}

\begin{proof}
We only need to prove the second inequality in \eqref{eq:small-time-bound},
which follows \cite[Section 3.2]{QSTY}.
Let $\rho\in C_c^\infty(\R^n,[0,\infty))$ be a standard mollifier, that is,
\[
\int_{\R^n}\rho(x)\,dx=1, \qquad {\rm supp}(\rho) \subset B(0,1), \qquad \rho(x) = \rho(-x)\text{ for }x\in \R^n,
\]
We set $\rho^\ep:=\frac{1}{\ep^n} \rho(\frac{\cdot}{\ep})$ and $g^{\ep}:=\rho^\ep * g$. Then, $g^\ep\in C^2(\R^n)$ and we have the bounds
\begin{equation}\label{eq:cor-nd-1}
\|g^\ep -g\|_{L^\infty(\R^n)} \leq C\ep,\qquad \|Dg^\ep\|_{L^\infty(\R^n)} + \ep \|D^2g^\ep\|_{L^\infty(\R^n)} \leq C.
\end{equation}
Let $\tilde u^\ep$ denote the viscosity solution to
\begin{equation}\label{1.1with gep}
\begin{cases}
 \tilde u_t^\ep+H\left(\frac{x}{\ep},D\tilde u^\ep\right)=\ep \Delta \tilde u^\ep \qquad &\text{in} \ \R^n \times (0,\infty),\\
\tilde u^\ep(x,0)=g^{\ep}(x) \qquad &\text{on} \ \R^n.
\end{cases} 
\end{equation}
By \eqref{eq:cor-nd-1} and the comparison principle, we deduce
\begin{equation}\label{eq:cor-nd-2}
\|\tilde u^\ep -u^\ep\|_{L^\infty(\R^n\times[0,\infty))} \leq C\ep.
\end{equation}
On the other hand, in view of \eqref{eq:cor-nd-1}, we have 
\[
\left|H(\frac{x}{\ep},Dg^\ep(x)) - \ep \Delta g^\ep(x)\right| \leq C \qquad \text{ for $x\in \R^n$},
\]
which yields that $(x,t)\mapsto g^\ep(x) + Ct$ is a supersolution to \eqref{1.1with gep}, and $(x,t)\mapsto g^\ep(x) - Ct$ is a subsolution to \eqref{1.1with gep}.
By the comparison principle,
\[
g^\ep(x) - Ct \leq \tilde u^\ep(x,t) \leq g^\ep(x) + Ct \qquad \text{ for all } (x,t) \in \R^n\times[0,\infty),
\]
Thus, for $(x,t)\in \R^n\times [0,\infty)$,
\[
|u^\ep(x,t)-g(x)|\le |\tilde u^\ep(x,t)-g(x)|+\|\tilde u^\ep -u^\ep\|_{L^\infty(\R^n\times[0,\infty))} \leq C(t+\ep).
\]

\end{proof}

We are ready to prove our main results.

\begin{proof}[{\bf Proof of Theorem \ref{theo:main}}]
Since replacing \(V\) by \(V+c\) changes both \(u^\ep\) and \(u\) by the same additive term \(-ct\), we may normalize
\[
\min_{\T^n}V=0.
\]
Then \(\|V\|_{L^\infty(\T^n)}\le \sqrt{n}\|DV\|_{L^\infty(\T^n)}\), so every constant below may be taken to depend only on \(n\), \(\|DV\|_{L^\infty(\T^n)}\), \(\|Dg\|_{L^{\infty}(\R^n)}\).

\medskip

Fix \((x_0,t_0)\in\R^n \times [0,\infty)\).
If $t_0 \leq \ep$, then by Lemma \ref{lem:bound-small-time}, we have
\[
|u^\ep(x_0,t_0)-u(x_0,t_0)| \leq C(t_0+\ep) \leq C\ep,
\]
which gives us the desired conclusion.

\medskip

We now consider the case $t_0 \geq \ep$.
Let $\tilde \ep = \ep/t_0 \in (0,1]$ and
\[
w^{\tilde \ep}(x,t):=\frac{1}{t_0}u^\ep(x_0+xt_0,tt_0),\qquad
\tilde g(x):=\frac{1}{t_0}g(x_0+xt_0),\qquad
\tilde V(y):=V\!\left(y+\frac{x_0}{\ep}\right).
\]
Then \(w^{\tilde \ep}\) solves the same viscous Hamilton--Jacobi equation 
\begin{equation*}
\begin{cases}
w_t^{\tilde \ep} + {1\over 2}|Dw^{\tilde\ep}|^2+\tilde V\left({x\over \tilde \ep}\right)={\tilde \ep\over 2} \Delta w^{\tilde \ep}
\qquad &\text{in }\R^n\times(0,\infty),\\
w^{\tilde \ep}(x,0)=\tilde g(x)
\qquad &\text{on }\R^n.
\end{cases}
\end{equation*}
And the corresponding homogenized solution is \(w(x,t)=\tfrac{1}{t_0}u(x_0+xt_0,tt_0)\). 
Since \(\tilde V\) is again \(\Z^n\)-periodic, \(\|D\tilde V\|_{L^\infty}=\|DV\|_{L^\infty}\), \(\|D\tilde g\|_{L^\infty}=\|Dg\|_{L^\infty}\), and the effective Hamiltonian is unchanged by spatial translation of the potential, the argument below applies verbatim with the same constants. Hence it suffices to establish the estimate at \((x_0,t_0)=(0,1)\).

Denote by
\[
h(y):=g(y)+\overline L\left(-y\right) \qquad \text{ for } y\in\mathbb{R}^n.
\]
By the Hopf--Lax formula (see \cite[Chapter 2]{Tran}), 
\[
u(0,1)=\min_{\mathbb{R}^n} h = h(\overline y),
\]
where $\overline y\in\mathbb{R}^n$ is a global minimum point of $h$.
 Since $g\in W^{1,\infty}(\mathbb{R}^n)$ and $\overline L$ grows quadratically thanks to Corollary \ref{cor:L bar},
there exists $C=C(\|Dg\|_{L^\infty},\|DV\|_{L^\infty},n)$ such that
 \[
 |\overline y| \leq C.
 \]
Indeed, \(h(0)\le g(0)\), while the lower quadratic bound in Corollary~\ref{cor:L bar} and the Lipschitz continuity of \(g\) imply \(h(y)\to+\infty\) as \(|y|\to\infty\).
Furthermore, there exists $A=A(\|Dg\|_{L^\infty},\|DV\|_{L^\infty})>0$ such that
\begin{equation}\label{eq:h-Lip}
|h(y)-h(\overline y)|\leq A |y-\overline y|  \qquad  \text{ for $y\in B(\overline y, 20C)$}.
\end{equation}
Indeed, since \(\overline L\in C^1(\R^n)\) and \(q\mapsto D\overline L(q)\) is bounded on the compact set \(\overline B(0,20C)\), the Lipschitz constant of \(h\) on \(B(\overline y,20C)\) is bounded only in terms of \(\|Dg\|_{L^\infty}\), \(\|DV\|_{L^\infty}\), and \(n\).

Recall that the Hopf--Cole representation \eqref{eq:u-ep-formula} and the rough kernel bound \eqref{eq:rough-bound-k} have already been established just before the proof.
Let $V_0=\ol B(0,20C)$.
Since \(K(s,0,-qs)\) is strictly positive and continuous in \((s,q)\) for \(s>0\), and since \(\overline L\) is continuous, the quantity
\[
s^{n/2}e^{s\overline L(q)}K(s,0,-qs)
\]
is bounded above and below by positive constants on the compact set \([1,t_0]\times V_0\). Combining this with Proposition~\ref{thm:ballistic-compact}, and enlarging the constants if necessary, we obtain positive constants \(c_\tau\) and \(C_\tau\) such that
\begin{equation}\label{eq:ballistic-tau}
\frac{c_\tau}{s^{n/2}}e^{-s\overline L(q)}
\le
K(s,0,-qs)
\le
\frac{C_\tau}{s^{n/2}}e^{-s\overline L(q)}
\qquad\text{for all }q\in V_0,\ s\ge 1.
\end{equation}
In particular, \eqref{eq:ballistic-tau} applies with \(s=1/\ep\) for every  \(\ep\in(0,1]\).

Next, the lower quadratic bound in Corollary~\ref{cor:L bar}, together with the Lipschitz continuity of \(g\) and the estimate \(|\overline y|\le C\), implies that after enlarging \(C\) if necessary, one has
\[
h(y)-h(\overline y)\ge \frac{1}{4}|y-\overline y|^2
\qquad\text{whenever }|y-\overline y|\ge 10C.
\]
Since \(|\overline y|\le C\), this yields
\begin{equation}\label{eq:tail-gap}
h(y)-h(\overline y)\ge \frac{|y|^2}{8}
\qquad\text{for }|y|\ge 10C.
\end{equation}
Using \eqref{eq:rough-bound-k}, when $C$ is sufficiently large, 

\[
\int_{\R^n\setminus B(0,10C)}K\!\left( \frac{1}{\ep},0, \frac{y}{\ep}\right)e^{-{g(y)\over \ep}}\,dy
\le
{1\over 4}\int_{B(0,1)}K\!\left( \frac{1}{\ep},0, \frac{y}{\ep}\right)e^{-{g(y)\over \ep}}\,dy
\]
Accordingly,
\[
\int_{\R^n}K\!\left( \frac{1}{\ep},0, \frac{y}{\ep}\right)e^{-{g(y)\over \ep}}\,dy=r_\ep\int_{B(0,10C)}K\!\left( \frac{1}{\ep},0, \frac{y}{\ep}\right)e^{-{g(y)\over \ep}}\,dy
\]
for $r_\ep\in \left[1,{4\over 3}\right]$.

Hence, replacing \(\R^n\) by \(B(0,10C)\) changes \(u^\ep(0,1)\) by at most \(O(\ep)\), which is absorbed in the final error term.

Owing to \eqref{eq:u-ep-formula} and \eqref{eq:ballistic-tau}, we have
\begin{align*}
    u^{\ep}(0,1)&=-\ep \log\left(\ep^{-n}\int_{\R^n}K\left( {1\over \ep},0, {y\over \ep}\right)e^{-{g(y)\over \ep}}\,dy\right)\\
    &=-\ep \log\left(\ep^{-n}\int_{B(0,10C)}K\left( {1\over \ep},0, {y\over \ep}\right)e^{-{g(y)\over \ep}}\,dy\right)+O(\ep)\\
    &=-\ep\log \left(\ep^{-{n\over 2}}\int_{B(0,10C)}e^{-{\overline L(-y)+g(y)\over \ep}}\,dy\right)+O(\ep).
\end{align*}
Therefore,
\begin{equation}\label{eq:4.14}
u^\varepsilon(0,1)-u(0,1)
=
-\varepsilon \log \left(
\ep^{-{n\over 2}}\int_{B(0,10C)}
e^{-\frac{h(y)-h(\overline y)}{\varepsilon}}\,dy
\right)+O(\ep).
\end{equation}

For $\varepsilon\in(0,1]$, we use \eqref{eq:h-Lip} to see that
\[
C_1
\geq
\int_{B(0,10C)}
e^{-\frac{h(y)-h(\overline y)}{\varepsilon}}\,dy
\ge
\int_{B(\overline y, 9C)}
e^{-\frac{A|y-\overline y|}{\varepsilon}}\,dy\ge C_2\varepsilon^n
\]
for some positive constants $C_1=C_1(n)$ and $C_2=C_2(\|Dg\|_{L^\infty},\|DV\|_{L^\infty},n)$. 

Hence, 
\[
|u^\ep(0,1)-u(0,1)| \leq  \ep \left(\frac{n}{2}|\log \ep|+ C\right).
\]

\end{proof}

\subsection{Proof of Theorem \ref{theo:main2}}
Recall that, by the Hopf--Lax formula,
\[
u(x,1)=\min_{y\in\R^n}\{g(y)+\overline L(x-y)\}
\]
Therefore, $u(\cdot,1)$ is semiconcave. 
For each $x\in \R^n$, denote by $y_x\in \R^n$ such that
\[
u(x,1)=g(y_x)+\overline L(x-y_x).
\]
If $u(\cdot,1)$ is differentiable at $x$, we have that
\[
Du(x,1)=D\overline L(x-y_x).
\]
Accordingly, 
\begin{equation}\label{eq:mini-point}
   y_x=x-D\overline H(Du(x,1)) .
\end{equation}
Write
\[
h_x(y)=g(y)+\overline L(x-y).
\]

The following lemma was supplied to us by ChatGPT.  
If the initial data $g$ is $C^2$, it can be established by adapting the proof of \cite[Proposition 4.4(ii)]{QSTY}. 
See Remark \ref{rmk:smoothg} for details. 

\begin{lemma}\label{lem:Q-Nondeneracy}  
Suppose that $u(\cdot,1)$ is twice differentiable at $x=x_0$ and write
\[
y_{x_0}=y_0 \qquad \mathrm{ and } \qquad h_{x_0}=h_0.
\]
Owing to (\ref{eq:mini-point}), $y_0$ is the unique minimum point of $h_0$.
Then, there exist $\delta_{x_0}, r_{x_0}>0$ such that
\[
h_{0}(y)\geq h_{0}(y_0)+\delta_{x_0} |y-y_0|^2 \qquad \text{ for } y\in B(y_0, r_{x_0}).
\] 
\end{lemma}
\begin{proof} 
Since this is a local conclusion, without loss of generality, we may assume that 
\[
I_n\leq D^2 \overline L\leq K I_n
\]
for a constant $K>0$.

Set
\[
\phi(y):=h_0(y)-h_0(y_0)>0 \qquad \text{ for $y\not=y_0$.}
\]
We argue by contradiction. Suppose that the conclusion fails. Then there exists a sequence
\(y_k\to y_0\) such that $y_k\not=y_0$ and
\[
\phi(y_k)=o(|y_k-y_0|^2).
\]
Write
\[
d_k:=y_k-y_0,\qquad \phi_k:=\phi(y_k).
\]

Since \(u(\cdot,1)\) is twice differentiable at \(x_0\),
\[
u(x,1)=u(x_0,1)+Du(x_0,1)\cdot(x-x_0)+O(|x-x_0|^2)
\qquad\text{as }x\to x_0.
\]
Also,
\[
Du(x_0,1)=D \overline L(x_0-y_0).
\]
Set
\[
p_0:=D\overline L(x_0-y_0),\qquad q_k:=p_0-D\overline L(x_0-y_k).
\]
Because \(I_n\le D^2\overline L\le KI_n\), we have
\[
|d_k| \leq  |q_k|\le K|d_k|.
\]

Now define
\[
s_k:=\frac{2\phi_k}{|d_k|},
\qquad
x_k:=x_0+s_k\frac{q_k}{|q_k|}.
\]
Since \(\phi_k=o(|d_k|^2)\), it follows that \(s_k=o(|d_k|)\), hence \(x_k\to x_0\).

Using \(y_k\) as a competitor in the definition of \(u(x_k,1)\), we obtain
\[
u(x_k,1)\le g(y_k)+\overline L(x_k-y_k).
\]
Therefore,
\[
u(x_k,1)-u(x_0,1)-p_0\cdot(x_k-x_0)
\le \phi_k+\Big(\overline L(x_k-y_k)-\overline L(x_0-y_k)-p_0\cdot(x_k-x_0)\Big).
\]
By Taylor's theorem and the bound \(D^2\overline L\le KI_n\),
\[
\overline L(x_k-y_k)-\overline L(x_0-y_k)
\le D\overline L(x_0-y_k)\cdot(x_k-x_0)+{K\over 2}|x_k-x_0|^2.
\]
Hence
\[
u(x_k,1)-u(x_0,1)-p_0\cdot(x_k-x_0)
\le \phi_k+\big(D\overline L(x_0-y_k)-p_0\big)\cdot(x_k-x_0)+{K\over 2}|x_k-x_0|^2.
\]
Since \(D\overline L(x_0-y_k)-p_0=-q_k\) and \(x_k-x_0=s_k q_k/|q_k|\), this yields
\[
u(x_k,1)-u(x_0,1)-p_0\cdot(x_k-x_0)
\le \phi_k-s_k|q_k|+{K\over 2}s_k^2.
\]
Using \(|q_k|\ge |d_k|\) and \(s_k=2\phi_k/|d_k|\), we get
\[
u(x_k,1)-u(x_0,1)-p_0\cdot(x_k-x_0)
\le -\phi_k+\frac{2K\phi_k^2}{|d_k|^2}.
\]
Since \(\phi_k=o(|d_k|^2)\), it follows that for \(k\) large,
\[
u(x_k,1)-u(x_0,1)-p_0\cdot(x_k-x_0)\le -\frac12\phi_k.
\]

On the other hand, since \(u\) is twice differentiable at \(x_0\),
\[
u(x_k,1)-u(x_0,1)-p_0\cdot(x_k-x_0)=O(|x_k-x_0|^2)=O(s_k^2).
\]
But
\[
s_k^2=\frac{4\phi_k^2}{|d_k|^2}=o(\phi_k),
\]
which contradicts the previous inequality since $\phi_k>0$. This proves the claim.
\end{proof}

\begin{remark}\label{rmk:smoothg} 
If $g$ is $C^2$, the above lemma can be proved by a more straightforward method. 
The notation ``$\cdot$" below represents matrix multiplications.  
In fact, at $y=y_x$,
\[
D^2 h(y_x)=D^2g(y_x)+D^2\overline {L}(x-y_x)\geq 0.
\]
 Also
 \[
y_x=x-D\overline H(Du(x,1)) \quad \Rightarrow  \quad Dy_x=I_n-D^2\overline H\cdot D^2u(x,1).
 \]
 Moreover,
 \[
 Du(x,1)=Dg(y_x)  \quad \Rightarrow  \quad  D^2u(x,1)=D^2g\cdot (I_n-D^2\overline H\cdot D^2u(x,1)).
 \]
 Accordingly,
 \begin{equation}\label{eq:equation-for-degeneracy}
(I_n+D^2g\cdot D^2\overline H)\cdot D^2u(x,1)=D^2g(y_x). 
\end{equation}
 If $D^2h(y_x)$ is not strictly positive, then there exists a unit vector $\xi$ such that
 \[
 \xi\cdot D^2h(y_x)=\xi\cdot (D^2g(y_x)+D^2\overline {L}(x-y_x))=0.
 \]
 Since $D^2\overline H(p)\cdot D^2\overline L(D\overline H(p))=I_n$, we have 
 \[
 \xi\cdot (D^2g(y_x)\cdot D^2\overline H(Du(x,1)+I_n)=0. 
 \]
 Combining with (\ref{eq:equation-for-degeneracy}), we derive that 
 \[
 \xi\cdot D^2g(y_x)=0,  
 \]
 which implies further that
 \[
\xi\cdot D^2\overline L(x-y_x)=0,
 \]
 which is impossible due to the strict convexity of $\overline L$. 
 ChatGPT helped us find an argument to remove the $C^2$ assumption. 
\end{remark}

\begin{proof}[{\bf Proof of Theorem \ref{theo:main2}}]
Without loss of generality, we assume $(x_0,t_0)=(0,1)$, and $u(\cdot,1)$ is twice differentiable at $0$.  

Since $g$ is locally semiconcave, $h=h_0$ is also locally semiconcave.  Therefore, combining with Lemma \ref{lem:Q-Nondeneracy}, there exists $\alpha_0=\alpha_{x_0}>0$ such that
 \[
\alpha_0|y-\overline y|^2\leq h(y)-h(\overline y)\leq {1\over \alpha_0}|y-\overline y|^2 \qquad \text{ for $y\in B(0,10C)$}. 
 \]
 Thanks to (\ref{eq:4.14}), we get
\[
|u^\ep(0,1)-u(0,1)| \leq  C_{0,1}\ep .
\]

\end{proof}

\begin{remark}
The local semiconcavity of $g$ is an essential assumption in the proof of Theorem \ref{theo:main2}. Indeed, the following example shows that the statement can fail without it.

    Let $V \equiv 0$, $g(x)=\min\{|x|,10\}$ for $x\in \R^n$.
    Then, $\overline H(p)=\frac{1}{2}|p|^2$ for $p\in \R^n$, and $\overline L(q)=\frac{1}{2}|q|^2$ for $q\in \R^n$.
    By direct computations, we see that
    \[
    u(x,1) = \frac{|x|^2}{2} \qquad \text{ for } x\in B(0,1).
    \]
    In particular, $u(\cdot,1)$ is smooth at $x_0=0$, and $\overline y=0$.
    Then, for $y\in B(0,10),$
    \[
    h(y)-h(0)=g(y)+\overline L(-y)=|y|+\frac{|y|^2}{2}.
    \]
    And,
\begin{align*}
u^\varepsilon(0,1)-u(0,1)
&=
-\varepsilon \log \left(
\ep^{-{n\over 2}}\int_{B(0,10C)}
e^{-\frac{h(y)-h(0)}{\varepsilon}}\,dy
\right)+O(\ep)\\
&=\frac{n}{2}\ep |\log \ep| + O(\ep).
\end{align*}
Thus, even if $u(\cdot,1)$ is smooth at $x_0=0$, the convergence rate of $u^\ep(0,1)-u(0,1)$ is still just $O(\ep |\log \ep|)$.
\end{remark}

\section{Proof of Proposition \ref{thm:ballistic-compact}}
\label{sec:ballistic}

 For $p\in \R^n$,  let $v_p$ be the unique $\Z^n$-periodic solution to the cell problem
\begin{equation}\label{eq:cell-p}
-{1\over 2}\Delta v_p+{1\over 2}|p+Dv_p|^2+V(x)=\overline H(p) \qquad \text{ in $\R^n$}
\end{equation}
subject to $v_p(0)=0$.
\medskip

Throughout this section, $e^{t\mathcal{L}}g$ represents the solution to
\[
\begin{cases}
u_t=\mathcal{L}u={1\over 2}\Delta u+Vu\qquad &\text{ in $\R^n\times (0,\infty)$},\\[3mm]
u(x,0)=g(x) \qquad &\text{ on $\R^n$}.
\end{cases}
\]

\subsection{Doob $h$-transform}\label{subsec:doob-h}

\begin{lemma} 
Let  $h_p(x)=e^{-p\cdot x-v_p(x)}$ for $x\in \R^n$. 
Then,
    \[
K(t,x,y)=e^{t\overline H(p)}{h_p(x)\over h_p(y)}\tilde p_p(t,x,y) \qquad \text{ for } (t,x,y)\in (0,\infty)\times \R^n\times \R^n.
\]
Here, $\tilde p_p(t,x,y)$ is the fundamental solution to $\partial_t - \mathcal{G}_p$, where
\begin{equation}\label{eq:B-PDE}
\mathcal{G}_pf={1\over 2}\Delta f+b_p(x)\cdot D f.
\end{equation}
And, $b_p(x)=-p-Dv_p(x)$, which is $\Z^n$-periodic.
\end{lemma}

\begin{proof}
Denote by $r_p(y)=e^{-v_p(y)}$. Then
\[
\frac{1}{2}\Delta r_p -p\cdot Dr_p + \left(\frac{1}{2}|p|^2 + V(y) - \overline H(p)\right) r_p = 0 \qquad \text{ in } \T^n,
\]
and $r_p(0)=1$. 
Also, $h_p(x)=e^{-p\cdot x} r_p(x)$ for $x\in\R^n$.

Let
\[
\varphi(x,t)=e^{t\overline H(p)} h_p(x) = \exp\left(t \overline H(p)-p\cdot x - v_p(x)\right).
\]
Then  
\[
\varphi=e^{t\mathcal{L}}h_p.
\]
Let $u(x,t)=(e^{t\mathcal{L}}f)(x)$.
Let
\[
v(x,t)=\frac{u(x,t)}{\varphi(x,t)}.
\]
Then,
\begin{align*}
    v_t&=\frac{u_t}{\varphi} -u \frac{\varphi_t}{\varphi^2},\\
    Dv&=\frac{Du}{\varphi}- u \frac{D\varphi}{\varphi^2},\\
    \Delta v&=\frac{\Delta u}{\varphi} - u \frac{\Delta \varphi}{\varphi^2}-2\left(\frac{Du}{\varphi}- u \frac{D\varphi}{\varphi^2} \right)\cdot \frac{D\varphi}{\varphi}\\
    &=\frac{\Delta u}{\varphi} - u \frac{\Delta \varphi}{\varphi^2}-2 Dv \cdot (-p-Dv_p).
\end{align*}
Hence,
\begin{align*}
    v_t -\frac{1}{2}\Delta v &=\frac{u_t-\frac{1}{2}\Delta u}{\varphi} -u \frac{\varphi_t-\frac{1}{2}\Delta \varphi}{\varphi^2}+Dv \cdot (-p-Dv_p)\\
    &=-\frac{Vu}{\varphi} +u\frac{V\varphi}{\varphi^2}+Dv \cdot (-p-Dv_p)=Dv \cdot (-p-Dv_p).
\end{align*}
Thus,
\[
v_t -\frac{1}{2}\Delta v=Dv \cdot (-p-Dv_p)=b_p\cdot Dv.
\]
Hence, 
\[
v(x,t)=\int_{\R^n}\tilde p_p(t,x,y){f(y)\over h_p(y)}\,dy.
\]
We conclude that
\[
\frac{1}{e^{t \overline H(p)}h_p(x)} \int_{\R^n}K(t,x,y)f(y)\,dy= \int_{\R^n}\widetilde p_p(t,x,y)\frac{1}{h_p(y)} f(y)\,dy.
\]
Since $f$ is arbitrary, 
\[
K(t,x,y)= e^{t \overline H(p)}\frac{h_p(x)}{h_p(y)} \widetilde p_p(t,x,y).
\]
\end{proof}

 \begin{remark}[Doob $h$-transform from the probabilistic point of view]
The PDE computation above has a natural probabilistic interpretation. Let $(X_t)_{t\ge 0}$ be Brownian motion in $\R^n$ under $\mathbb P_x$, started from $x$. Since the linear operator here is
\[
\mathcal L=\frac12\Delta +V,
\]
the Feynman--Kac formula gives
\[
(e^{t\mathcal L}f)(x)
=
\mathbb E_x\!\left[
\exp\!\left(\int_0^t V(X_s)\,ds\right)f(X_t)
\right].
\]
Thus, the zeroth-order term $+V$ appears probabilistically as a multiplicative path weight. Depending on the sign convention, such a factor may be viewed as \textit{growth, discounting, or killing} after a standard renormalization by constants; in any case, the resulting semigroup is no longer conservative, since it does not preserve constants.

Now $h_p$ is a positive eigenfunction satisfying
\[
\mathcal L h_p=\overline H(p)\,h_p.
\]
By It\^o's formula, the process
\[
M_t^{(p)}
:=
e^{-t\overline H(p)}
\exp\!\left(\int_0^t V(X_s)\,ds\right)
\frac{h_p(X_t)}{h_p(x)}
\]
is a martingale under $\mathbb P_x$. Hence it defines a new probability measure on path space by
\[
\frac{d\widetilde{\mathbb P}_{x,p}}{d\mathbb P_x}\Big|_{\mathcal F_t}
=
M_t^{(p)}.
\]
Under this new measure, the coordinate process is a genuine diffusion with generator
\[
\mathcal G_p f
=
\frac12\Delta f+b_p(x)\cdot Df,
\qquad
b_p(x)=D\log h_p(x)=-p-Dv_p(x).
\]
In this sense, the Doob $h$-transform converts the non-conservative Feynman--Kac weighted Brownian motion into a conservative Markov diffusion: the original zeroth-order potential term is absorbed into the change of measure and reappears as the drift $b_p$.

The kernel identity
\[
K(t,x,y)=e^{t\overline H(p)}\frac{h_p(x)}{h_p(y)}\widetilde p_p(t,x,y)
\]
is exactly the density-level version of this renormalization. Equivalently, the PDE quotient transformation
\[
u=\varphi\,v,
\qquad
\varphi(x,t)=e^{t\overline H(p)}h_p(x),
\]
is the analytic counterpart of the same probabilistic mechanism: dividing by the positive solution $\varphi$ removes the zeroth-order term from the equation and produces a drift-diffusion equation for $v$; see also \cite{Pinsky}.
\end{remark}

We have the following large-time bound along the ballistic ray, which is a corollary of \cite[Theorem 1.3]{Norris1997}.

\begin{proposition}[Local limit bound, two-sided]\label{prop:LLT}
Fix a compact set $P\subset\R^n$. There exist $t_0\ge1$ and constants $0<C_1\le C_2<\infty$ such that
for every $p\in P$ and every $t\ge t_0$,
\begin{equation}\label{eq:LLT}
{C_1}{t^{-{n/2}}}\le \widetilde p_p\bigl(t,0,-D\overline H(p) t)\le {C_2}{t^{-{n/2}}}.
\end{equation}
\end{proposition}

A sharper local limit asymptotic was announced in \cite{Agmon} and later proved in the more general branching-diffusion setting in \cite[Proposition~2.1]{HKN20}. Proposition~\ref{prop:LLT} extracts only the uniform two-sided estimate needed for the proof of Proposition~\ref{thm:ballistic-compact}. Appendix \ref{sec:appen} later records a finer asymptotic, formulated for the present family of Doob-transformed periodic diffusions.

\begin{proof}
Recall that $\pi_p$ solves
\[
\begin{cases}
\mathcal{G}_p^* \pi_p=-\frac{1}{2}\Delta \pi_p+\operatorname{div}(\pi_p b_p)=0 \qquad \text{ in } \T^n,\\
\int_{\T^n} \pi_p(x)\,dx=1.
\end{cases}
\]
To apply the result from \cite{Norris1997},   we first rewrite the operator $\mathcal{G}_p$ into the so called ``canonical form" (\cite[equation (1.4)]{Norris1997}):
\[
\mathcal{G}_pf=\operatorname{div}_{\pi_p}\!\left(\left(\frac12 I+\beta(x)\right)D f\right)
+ \frac{\overline b}{\pi_p(x)} \cdot D f.
\]
Here, 
\[
\operatorname{div}_{\pi_p} X
=
\frac{1}{\pi_p}\,\operatorname{div}(\pi_p X).
\]
 denotes the divergence associated with $\pi_p$. 

Also,
\[
\overline b=- D\overline H(p)= \int_{\mathbb T^n} b_p(x)\pi_p(x)\,dx,
\]
and \(\beta_p\) is a periodic antisymmetric matrix field chosen so that
\[
J(x)=\pi_p(x)b_p(x)-\frac12 D \pi_p(x)=\overline b+\operatorname{div}(\pi_p\beta_p).
\]
More explicitly, \(\beta_p={K(x)\over \pi_p(x)}\) and
\[
K_{ij}(x) := (\psi_i)_{x_j} -  ( \psi_j)_{x_i},
\]
where $\psi$ is the unique mean-zero periodic solution to  the following  Poisson equation 
\[
-\Delta \psi = -F
\qquad \text{ in } \mathbb{T}^n,
\]
where
\[
F(x) := J(x) -\int_{\T^n}J(y)\,dy=J(x)-\overline b.
\]
Since $\mathrm{div}(J)=0$,  $\mathrm{div}(\psi)=0$. Moreover, because \(p\mapsto v_p\) is continuous in \(C^{2,\alpha}(\T^n)\) on compact subsets of \(\R^n\), the maps \(p\mapsto \pi_p\), \(p\mapsto \overline b\), and \(p\mapsto \beta_p\) are continuous on \(P\). Since \(P\) is compact, the canonical-form coefficients above remain in a compact subset of the class covered by \cite[Theorem~1.3]{Norris1997}. Therefore, the constants in \cite[Theorem 1.3]{Norris1997} may be chosen uniformly for \(p\in P\), and the conclusion follows.
\end{proof}

\begin{remark}
The estimate in Proposition \ref{prop:LLT} is a ballistic local central limit type result, or equivalently a sharp long-time heat-kernel asymptotic, for the periodic diffusion generated by $\mathcal G_p$; compare \cite{Bhattacharya1985,Norris1997}. The Appendix explains the same Gaussian mechanism from a Bloch--Floquet theory/spectral perspective; see also \cite{Kuchment1993,ConcaVanninathan1997,KotaniSunada2000}.
\end{remark}

\medskip

\begin{proof}[{\bf Proof of Proposition \ref{thm:ballistic-compact}}]

Fix a compact set $V_0\subset \R^n$. For each $q\in V_0$, choose $p\in \R^n$ such that
\[
q=D\overline H(p).
\]
Since $\overline H$ is strictly convex and superlinear in the quadratic setting (see, for instance, \cite[Chapter~2]{Tran}), the set
\[
P:=\{p\in \R^n:\ D\overline H(p)\in V_0\}
\]
is compact.

By Lemma~\ref{lem:doob-h-transform}, for every $t>0$ and $y\in \R^n$,
\[
K(t,0,y)=e^{t\overline H(p)}\frac{h_p(0)}{h_p(y)}\,\tilde p_p(t,0,y),
\]
where
\[
h_p(x)=e^{-p\cdot x-v_p(x)}.
\]
Taking $y=-qt=-D\overline H(p)t$, we obtain
\[
K(t,0,-qt)
=
e^{t\overline H(p)}
\exp\bigl(-p\cdot qt+v_p(0)-v_p(-qt)\bigr)
\tilde p_p(t,0,-qt).
\]
Equivalently,
\[
K(t,0,-qt)
=
\exp\Bigl(t(\overline H(p)-p\cdot q)\Bigr)
\exp\bigl(v_p(0)-v_p(-qt)\bigr)
\tilde p_p(t,0,-qt).
\]

Now, since $q=D\overline H(p)$, by the Legendre duality relation (see \cite[Chapter 2]{Tran}), we have
\[
\overline L(q)=q\cdot p-\overline H(p),
\]
and hence
\[
\overline H(p)-p\cdot q=-\overline L(q).
\]
Therefore,
\[
K(t,0,-qt)
=
e^{-t\overline L(q)}
\exp\bigl(v_p(0)-v_p(-qt)\bigr)
\tilde p_p(t,0,-qt).
\]

Because $v_p$ is $\Z^n$-periodic, $v_p$ is bounded on $\R^n$. Moreover, since $p\in P$ and $P$ is compact, there exists $C\ge 1$ such that
\[
C^{-1}\le
\exp\bigl(v_p(0)-v_p(-qt)\bigr)
\le C
\]
for all $p\in P$, $t\ge 1$.

On the other hand, by Proposition~\ref{prop:LLT},
\[
C_1 t^{-n/2}\le \tilde p_p(t,0,-qt)\le C_2 t^{-n/2}
\]
for all $p\in P$, $t\ge t_0$.

Combining the last two estimates yields
\[
c\, t^{-n/2} e^{-t\overline L(q)}
\le
K(t,0,-qt)
\le
C\, t^{-n/2} e^{-t\overline L(q)}
\]
for all $q\in V_0$, $t\ge t_0$, where $0<c\le C<\infty$ depend only on $V_0$.

This is exactly the conclusion of Proposition~\ref{thm:ballistic-compact}.
\end{proof}


\appendix

\section{Large-time Gaussian asymptotics via Bloch--Floquet decomposition}\label{sec:appen}

In this appendix, we give a self-contained proof of the large-time Gaussian asymptotics for the heat kernel in a periodic setting. 
A closely related asymptotic was announced in \cite{Agmon} and later established in the more general branching-diffusion setting in \cite[Proposition~2.1]{HKN20}. This more delicate estimate is not needed for the proof of our main theorems. Nevertheless, for the convenience of the reader, and because both the result and its proof may be useful in other contexts, we provide a direct Bloch--Floquet argument adapted to the operator considered here.
Bloch--Floquet decompositions for periodic differential operators are classical; see \cite[Chapters~2--3]{Kuchment1993} and \cite{ConcaVanninathan1997}. 
The long-time off-diagonal Gaussian asymptotics are also closely related to perturbative spectral analysis near the principal Bloch branch; compare \cite{KotaniSunada2000}.

\subsection{Statement of the theorem}

Let
\[
A=\frac12\Delta+b(x)\cdot D
\qquad \text{on }\R^n,
\]
where \(b\in \Lip(\R^n,\R^n)\) is \(\Z^n\)-periodic. Let \(p(t,x,y)\) denote the heat kernel of \(A\), that is, for
\[
\begin{cases}
u_t=Au=\frac12\Delta u+b(x)\cdot Du
\qquad &\text{in }\R^n\times(0,\infty),\\[2mm]
u(x,0)=g(x)
\qquad &\text{on }\R^n,
\end{cases}
\]
we have
\[
u(x,t)=\int_{\R^n}p(t,x,y)g(y)\,dy.
\]

Let \(m\) be the invariant density on \(\T^n=\R^n/\Z^n\), normalized by
\begin{equation}\label{eq:A1-invariant-density}
A^*m
=
-\frac12\Delta m+\operatorname{div}(bm)=0
\qquad \text{in }\T^n,
\qquad
\int_{\T^n}m(x)\,dx=1.
\end{equation}
We keep the notation \(m\) for its \(\Z^n\)-periodic extension to \(\R^n\).
Define the effective drift
\[
\overline b=\int_{\T^n}b(x)m(x)\,dx.
\]
For each \(j=1,\dots,n\), let \(\chi_j\) be the unique periodic solution of
\begin{equation}\label{eq:A1-corrector}
A\chi_j=-(b_j-\overline b_j)
\qquad \text{in }\T^n,
\qquad
\int_{\T^n}\chi_j\,m\,dx=0.
\end{equation}
The effective diffusion matrix \(Q=(Q_{jk})\) is defined by
\[
Q_{jk}
=
\int_{\T^n}(e_j+D\chi_j)\cdot(e_k+D\chi_k)\,m(x)\,dx.
\]

\begin{theorem}[Large-time Gaussian asymptotics]\label{thm:Large-time Gaussian asymptotics}
The matrix \(Q\) is symmetric positive definite. Moreover, there exist constants \(C,\eta>0\) such that for all \(t\ge 1\) and all \(x,y\in\R^n\),
\[
p(t,x,y)
=
m(y)\frac{1}{(2\pi t)^{n/2}\sqrt{\det Q}}
\exp\!\left(
-\frac{1}{2t}(y-x-\overline bt)\cdot Q^{-1}(y-x-\overline bt)
\right)
+R(t,x,y),
\]
where
\[
|R(t,x,y)|\le C\,t^{-(n+1)/2}+Ce^{-\eta t}.
\]
In particular,
\[
p(t,x,y)
=
m(y)\frac{1}{(2\pi t)^{n/2}\sqrt{\det Q}}
\exp\!\left(
-\frac{1}{2t}(y-x-\overline bt)\cdot Q^{-1}(y-x-\overline bt)
\right)
+o(t^{-n/2})
\]
uniformly in \(x,y\in\R^n\) as \(t\to\infty\).
\end{theorem}

\subsection{Proof of Theorem \ref{thm:Large-time Gaussian asymptotics}}

\begin{proof}
The proof of Theorem~\ref{thm:Large-time Gaussian asymptotics} is divided into five steps.

\medskip
\noindent
{\bf Step 1: Bloch representation of the heat kernel.}
Let \(B=[-\pi,\pi)^n\). For each \(\xi\in B\), define the fiber operator
\begin{equation}\label{eq:A1-fiber}
A_\xi
:=
e^{-i\xi\cdot x}Ae^{i\xi\cdot x}
=
\frac12(D+i\xi)^2+b(x)\cdot(D+i\xi)
\qquad \text{on }H^2(\T^n,\C).
\end{equation}
By the Bloch--Floquet decomposition for periodic differential operators, the Bloch transform \(\mathcal B\) diagonalizes \(A\) as the direct integral
\[
\mathcal B A\mathcal B^{-1}
=
\int_B^\oplus A_\xi\,d\xi;
\]
see \cite[Chapters~2--3]{Kuchment1993} and \cite{ConcaVanninathan1997}. Consequently,
\[
\mathcal B e^{tA}\mathcal B^{-1}
=
\int_B^\oplus e^{tA_\xi}\,d\xi.
\]
For every \(t>0\), the fiber semigroup \(e^{tA_\xi}\) admits a jointly continuous kernel \(K_\xi(t,\cdot,\cdot)\) on \(\T^n\times\T^n\) by standard parabolic regularity on the compact manifold \(\T^n\). Therefore
\begin{equation}\label{eq:A1-bloch-kernel}
p(t,x,y)
=
\frac{1}{(2\pi)^n}\int_B e^{i\xi\cdot(x-y)}K_\xi(t,\bar x,\bar y)\,d\xi,
\end{equation}
where \(\bar x,\bar y\) denote the classes of \(x,y\) in \(\T^n\).

\medskip
\noindent
{\bf Step 2: Principal spectral branch near \(\xi=0\).}
We next analyze the fiber family \(A_\xi\) near the untwisted operator \(A_0=A\).

\smallskip
\noindent
\emph{(i) The eigenvalue \(0\) of \(A\) is algebraically simple.}
Since \(A1=0\), the constant function \(1\) is a right eigenfunction of \(A\). Since \(m\) solves \eqref{eq:A1-invariant-density}, we also have
\[
\int_{\T^n}Af\,m\,dx=0
\qquad \text{for every }f\in C^2(\T^n),
\]
so \(m\) is a left eigenfunction corresponding to the eigenvalue \(0\).

We claim that \(\ker A=\mathrm{span}\{1\}\). Indeed, if \(Af=0\) and \(f\in C^2(\T^n)\), then the strong maximum principle for uniformly elliptic operators with bounded coefficients implies that \(f\) is constant; see, for instance, \cite[Chapter~2]{Pinsky}. Thus the geometric multiplicity of the eigenvalue \(0\) equals \(1\). To rule out generalized eigenvectors, suppose that \(Ag=1\) for some \(g\in C^2(\T^n)\). Integrating against \(m\) yields
\[
1
=
\int_{\T^n}1\cdot m\,dx
=
\int_{\T^n}Ag\,m\,dx
=
0,
\]
a contradiction. Hence the algebraic multiplicity of the eigenvalue \(0\) is also \(1\).

Since \(A\) generates a Markov semigroup on \(C(\T^n)\), we have \(s(A)=0\). We also claim that \(0\) is the only spectral value of \(A\) on the imaginary axis. Indeed, if \(A\varphi=\mu\varphi\) with \(\Re\mu=0\), then
\[
e^{tA}\varphi=e^{t\mu}\varphi
\qquad (t>0).
\]
Taking \(t=1\) and arguing exactly as in Step~3(i) below with \(\xi=0\), one sees that \(|\varphi|\) must be constant on \(\T^n\); positivity improvement of \(e^A\) then forces \(\varphi\) to be constant, and therefore \(\mu=0\). Since \(A\) has compact resolvent, it follows that there exists \(\eta_0>0\) such that
\begin{equation}\label{eq:A1-gap-at-zero}
\sigma(A)\setminus\{0\}\subset \{z\in\C:\ \Re z\le -4\eta_0\}.
\end{equation}

Therefore, by analytic perturbation theory for simple isolated eigenvalues \cite[Chapter~VII, \S1]{Kato}, there exist \(r_0>0\) and analytic branches
\[
\lambda(\xi),\qquad \Phi_\xi,\qquad \Psi_\xi,
\qquad |\xi|<r_0,
\]
such that
\[
A_\xi\Phi_\xi=\lambda(\xi)\Phi_\xi,
\qquad
A_\xi^*\Psi_\xi=\overline{\lambda(\xi)}\Psi_\xi,
\qquad
\lambda(0)=0,
\]
and
\[
\Phi_0\equiv 1,
\qquad
\Psi_0=m.
\]
We choose these branches so that
\begin{equation}\label{eq:A1-right-normalization}
\int_{\T^n}\Phi_\xi(x)m(x)\,dx=1
\qquad (|\xi|<r_0),
\end{equation}
and
\begin{equation}\label{eq:A1-left-right-normalization}
\int_{\T^n}\Psi_\xi(x)\Phi_\xi(x)\,dx=1
\qquad (|\xi|<r_0).
\end{equation}
Because the principal part is constant and the lower-order coefficients are Lipschitz, elliptic regularity upgrades the analytic branches above to \(C^{2,\alpha}(\T^n)\) for every \(\alpha\in(0,1)\). In particular, all \(O(|\xi|)\) and \(O(|\xi|^2)\) statements below are uniform on \(\T^n\).

\smallskip
\noindent
\emph{(ii) The first-order term.}
Write
\[
A_\xi
=
A+\sum_{j=1}^n \xi_jB_j+\frac12\sum_{j,k=1}^n \xi_j\xi_k C_{jk},
\qquad
B_j=i(D_j+b_j),
\qquad
C_{jk}=-\delta_{jk}I.
\]
Differentiating the eigenvalue equation \(A_\xi\Phi_\xi=\lambda(\xi)\Phi_\xi\) at \(\xi=0\) in the \(j\)-th direction gives
\[
A\bigl(\partial_{\xi_j}\Phi_\xi|_{\xi=0}\bigr)+B_j1=(\partial_{\xi_j}\lambda)(0).
\]
Pairing with \(m\) and using \(\int Af\,m\,dx=0\), we obtain
\[
(\partial_{\xi_j}\lambda)(0)
=
\int_{\T^n}B_j1\,m\,dx
=
i\int_{\T^n}b_jm\,dx
=
i\overline b_j.
\]
Set
\[
\nu_j:=\partial_{\xi_j}\Phi_\xi\big|_{\xi=0},
\qquad
\chi_j:=-i\,\nu_j.
\]
Since \eqref{eq:A1-right-normalization} implies \(\int_{\T^n}\nu_jm\,dx=0\), the differentiated eigenvalue equation becomes
\[
A\chi_j=-(b_j-\overline b_j),
\qquad
\int_{\T^n}\chi_jm\,dx=0,
\]
which is precisely \eqref{eq:A1-corrector}. Therefore
\begin{equation}\label{eq:A1-eigenfunction-expansion}
\Phi_\xi
=
1+i\sum_{j=1}^n\xi_j\chi_j+O(|\xi|^2),
\qquad
\Psi_\xi
=
m+O(|\xi|),
\end{equation}
uniformly on \(\T^n\).

\smallskip
\noindent
\emph{(iii) The quadratic term and the effective diffusion matrix.}
Introduce the stationary current
\[
J:=bm-\frac12Dm.
\]
Equation \eqref{eq:A1-invariant-density} is equivalent to
\[
\operatorname{div}J=0
\qquad \text{in }\T^n.
\]
For \(f,g\in C^2(\T^n)\), an integration by parts gives
\begin{equation}\label{eq:A1-bilinear-identity}
\int_{\T^n}f\,Ag\,m\,dx
=
-\frac12\int_{\T^n}Df\cdot Dg\,m\,dx
+\int_{\T^n}f\,J\cdot Dg\,dx.
\end{equation}

Differentiating the eigenvalue equation twice at \(\xi=0\) gives
\[
A\bigl(\partial_{\xi_j\xi_k}^2\Phi_\xi|_{\xi=0}\bigr)
+B_j\nu_k+B_k\nu_j+C_{jk}1
=
(\partial_{\xi_j\xi_k}^2\lambda)(0)
+(\partial_{\xi_j}\lambda)(0)\nu_k+(\partial_{\xi_k}\lambda)(0)\nu_j.
\]
Pairing with \(m\), using \(\int Af\,m\,dx=0\), the identities \((\partial_{\xi_j}\lambda)(0)=i\overline b_j\), \(\nu_j=i\chi_j\), and the normalization \(\int_{\T^n}\nu_jm\,dx=0\), we obtain
\begin{align}
(\partial_{\xi_j\xi_k}^2\lambda)(0)
&=
-\delta_{jk}
-\int_{\T^n}\Big(D_j\chi_k+D_k\chi_j
+(b_j-\overline b_j)\chi_k+(b_k-\overline b_k)\chi_j\Big)m\,dx.
\label{eq:A1-second-derivative-pre}
\end{align}
To simplify the last two terms, use \(A\chi_j=-(b_j-\overline b_j)\) together with \eqref{eq:A1-bilinear-identity}:
\[
\int_{\T^n}(b_j-\overline b_j)\chi_km\,dx
=
-\int_{\T^n}\chi_kA\chi_jm\,dx
=
\frac12\int_{\T^n}D\chi_k\cdot D\chi_j\,m\,dx
-\int_{\T^n}\chi_kJ\cdot D\chi_j\,dx.
\]
Adding the same identity with \(j\) and \(k\) interchanged and using \(\operatorname{div}J=0\), we get
\[
\int_{\T^n}\Big((b_j-\overline b_j)\chi_k+(b_k-\overline b_k)\chi_j\Big)m\,dx
=
\int_{\T^n}D\chi_j\cdot D\chi_k\,m\,dx,
\]
because
\[
\int_{\T^n}\chi_kJ\cdot D\chi_j\,dx+\int_{\T^n}\chi_jJ\cdot D\chi_k\,dx
=
\int_{\T^n}J\cdot D(\chi_j\chi_k)\,dx
=
-\int_{\T^n}\operatorname{div}J\,\chi_j\chi_k\,dx
=
0.
\]
Substituting this identity into \eqref{eq:A1-second-derivative-pre}, we find
\[
(\partial_{\xi_j\xi_k}^2\lambda)(0)
=
-\int_{\T^n}(e_j+D\chi_j)\cdot(e_k+D\chi_k)\,m\,dx
=
-Q_{jk}.
\]
Hence
\begin{equation}\label{eq:A1-lambda-expansion}
\lambda(\xi)
=
i\overline b\cdot \xi-\frac12\,\xi\cdot Q\xi+O(|\xi|^3)
\qquad (\xi\to 0).
\end{equation}

\smallskip
\noindent
\emph{(iv) Positivity of \(Q\) and the sign of \(\Re\lambda(\xi)\).}
The matrix \(Q\) is symmetric by definition. For \(\nu\in\R^n\), let \(\chi_\nu:=\sum_{j=1}^n\nu_j\chi_j\). Then
\[
\nu\cdot Q\nu
=
\int_{\T^n}|\nu+D\chi_\nu|^2m\,dx.
\]
Thus \(Q\) is nonnegative definite. If \(\nu\cdot Q\nu=0\), then \(D\chi_\nu\equiv -\nu\) on \(\T^n\), which is impossible unless \(\nu=0\), because \(\chi_\nu\) is periodic. Therefore \(Q\) is positive definite.

Since the linear term in \eqref{eq:A1-lambda-expansion} is purely imaginary, there exist \(c_0>0\) and \(r\in(0,r_0)\) such that
\begin{equation}\label{eq:A1-small-xi-damping}
\Re\lambda(\xi)\le -c_0|\xi|^2
\qquad \text{for }|\xi|<r.
\end{equation}

\medskip
\noindent
{\bf Step 3: Spectral gap and semigroup decomposition.}
We now show that the principal Bloch mode controls the large-time behavior of the fiber semigroup.

\smallskip
\noindent
\emph{(i) Spectral gap away from \(\xi=0\).}
Fix \(\tau=1\) and write \(T_\xi:=e^{A_\xi}\) on \(C(\T^n,\C)\). Let \(X_t^{\tilde x}\) solve
\[
dX_t=b(X_t)\,dt+dW_t,
\qquad
X_0^{\tilde x}=\tilde x,
\]
for a lift \(\tilde x\in\R^n\) of \(x\in\T^n\). A standard It\^o computation shows that
\begin{equation}\label{eq:A1-twisted-semigroup}
(T_\xi f)(x)
=
\E\!\left[
e^{i\xi\cdot(X_1^{\tilde x}-\tilde x)}\,f(X_1^{\tilde x}\bmod \Z^n)
\right].
\end{equation}
Consequently,
\begin{equation}\label{eq:A1-domination}
|T_\xi f|\le T_0|f|.
\end{equation}
The operator \(T_0\) is compact and positivity improving on \(C(\T^n)\), because the torus diffusion has a strictly positive continuous transition density at time \(1\).

We claim that
\begin{equation}\label{eq:A1-no-zero-away}
s(A_\xi)<0
\qquad \text{for every }\xi\in B\setminus\{0\},
\end{equation}
where \(s(A_\xi)=\sup\{\Re z:\ z\in\sigma(A_\xi)\}\). Since \(A_\xi\) has compact resolvent, the spectral bound is attained by an eigenvalue. Suppose, to the contrary, that \(s(A_\xi)=0\) for some \(\xi\neq 0\). Then there exist \(\mu\in\C\) and \(\varphi\in C(\T^n)\setminus\{0\}\) such that
\[
A_\xi\varphi=\mu\varphi,
\qquad
\Re\mu=0.
\]
Exponentiating the eigenvalue equation gives
\[
T_\xi\varphi=e^\mu\varphi,
\qquad
|e^\mu|=1.
\]
Choose \(x_0\in\T^n\) with \(|\varphi(x_0)|=\|\varphi\|_\infty\). Using \eqref{eq:A1-domination},
\[
\|\varphi\|_\infty
=
|T_\xi\varphi(x_0)|
\le
T_0|\varphi|(x_0)
\le
\|\varphi\|_\infty.
\]
Hence equality holds throughout. Since \(T_0\) is positivity improving, the equality \(T_0|\varphi|(x_0)=\|\varphi\|_\infty\) forces \(|\varphi|\equiv \|\varphi\|_\infty\) on \(\T^n\). Equality in the triangle inequality in \eqref{eq:A1-twisted-semigroup} then implies that the complex-valued random variable
\[
e^{i\xi\cdot(X_1^{\tilde x_0}-\tilde x_0)}\,
\varphi(X_1^{\tilde x_0}\bmod\Z^n)
\]
has almost surely constant argument. Since the diffusion is uniformly elliptic, \(X_1^{\tilde x_0}\) has a strictly positive density on \(\R^n\). Therefore continuity yields
\[
\varphi(y\bmod \Z^n)=Ce^{-i\xi\cdot y}
\qquad \text{for all }y\in\R^n
\]
for some \(C\neq 0\). Periodicity of \(\varphi\) now gives \(e^{-i\xi\cdot k}=1\) for every \(k\in\Z^n\), hence \(\xi\in 2\pi\Z^n\). Since \(\xi\in[-\pi,\pi)^n\) and \(\xi\neq 0\), this is impossible. Thus \eqref{eq:A1-no-zero-away} holds.

By compactness of \(B\setminus B_r\) and upper semicontinuity of the spectral bound for the compact-resolvent family \(\xi\mapsto A_\xi\), there exists \(\eta_1>0\) such that
\begin{equation}\label{eq:A1-far-gap}
s(A_\xi)\le -2\eta_1
\qquad \text{for every }\xi\in B\setminus B_r.
\end{equation}

\smallskip
\noindent
\emph{(ii) Spectral decomposition for the near fibers.}
By \eqref{eq:A1-gap-at-zero}, the eigenvalue \(0\) is separated from the rest of the spectrum of \(A\). After shrinking \(r\) if necessary, analytic perturbation theory shows that for every \(|\xi|<r\),
\[
|\lambda(\xi)|\le \eta_0
\qquad\text{and}\qquad
\sigma(A_\xi)\setminus\{\lambda(\xi)\}\subset\{z\in\C:\ \Re z\le -2\eta_0\}.
\]
Let \(\Gamma_0=\{z\in\C:\ |z|=2\eta_0\}\). Then the Riesz projection
\[
\Pi_\xi
=
\frac{1}{2\pi i}\int_{\Gamma_0}(z-A_\xi)^{-1}\,dz
\]
has rank one and depends analytically on \(\xi\). By \eqref{eq:A1-left-right-normalization},
\[
\Pi_\xi f
=
\Phi_\xi\int_{\T^n}\Psi_\xi(y)f(y)\,dy,
\]
so its kernel equals
\[
\Pi_\xi(x,y)=\Phi_\xi(x)\Psi_\xi(y).
\]

Let \(\Gamma_1\) be the boundary of the half-plane strip \(\{z\in\C:\ \Re z\le -\eta_0\}\), oriented positively around \(\sigma(A_\xi)\setminus\{\lambda(\xi)\}\). By the Dunford--Taylor functional calculus,
\[
e^{tA_\xi}
=
e^{t\lambda(\xi)}\Pi_\xi
+
N_\xi(t),
\qquad
N_\xi(t)
=
\frac{1}{2\pi i}\int_{\Gamma_1}e^{tz}(z-A_\xi)^{-1}\,dz.
\]
The resolvent is uniformly bounded on \(\Gamma_1\) for \(|\xi|<r\), because \(\Gamma_1\) stays a positive distance away from the spectrum. Hence
\begin{equation}\label{eq:A1-near-semigroup-op}
e^{tA_\xi}=e^{t\lambda(\xi)}\Pi_\xi+N_\xi(t),
\qquad
\|N_\xi(t)\|_{C(\T^n)\to C(\T^n)}\le Ce^{-\eta_0 t}
\qquad (|\xi|<r,\ t\ge 0).
\end{equation}
Similarly, \eqref{eq:A1-far-gap} yields
\begin{equation}\label{eq:A1-far-semigroup-op}
\|e^{tA_\xi}\|_{C(\T^n)\to C(\T^n)}\le Ce^{-\eta_1 t}
\qquad (\xi\in B\setminus B_r,\ t\ge 0).
\end{equation}

\smallskip
\noindent
\emph{(iii) Passage from operator bounds to kernel bounds.}
For \(t=1\), the kernels \(K_\xi(1,\cdot,\cdot)\) are jointly continuous on the compact set \(B\times\T^n\times\T^n\), so
\[
M_1:=\sup_{\xi\in B}\sup_{x,y\in\T^n}|K_\xi(1,x,y)|<\infty.
\]
For \(|\xi|<r\) and \(t\ge 1\), define
\[
R_\xi(t,x,y):=\bigl(N_\xi(t-1)[K_\xi(1,\cdot,y)]\bigr)(x).
\]
Then \eqref{eq:A1-near-semigroup-op} gives
\[
|R_\xi(t,x,y)|\le CM_1e^{-\eta_0(t-1)}\le Ce^{-\eta t}
\]
for some \(\eta\in(0,\min\{\eta_0,\eta_1\})\). Therefore
\begin{equation}\label{eq:A1-near-kernel}
K_\xi(t,x,y)
=
e^{t\lambda(\xi)}\Phi_\xi(x)\Psi_\xi(y)+R_\xi(t,x,y),
\qquad
|R_\xi(t,x,y)|\le Ce^{-\eta t}
\qquad (|\xi|<r,\ t\ge 1).
\end{equation}
Likewise, if \(\xi\in B\setminus B_r\) and \(t\ge 1\), then
\[
|K_\xi(t,x,y)|
=
\bigl|\bigl(e^{(t-1)A_\xi}[K_\xi(1,\cdot,y)]\bigr)(x)\bigr|
\le
Ce^{-\eta t}.
\]

\medskip
\noindent
{\bf Step 4: Reduction to the principal Gaussian integral.}
Set
\[
z:=y-x-\overline bt.
\]
By \eqref{eq:A1-bloch-kernel} and Step~3,
\[
p(t,x,y)=I_{\mathrm{near}}(t,x,y)+O(e^{-\eta t}),
\]
where
\[
I_{\mathrm{near}}(t,x,y)
:=
\frac{1}{(2\pi)^n}\int_{|\xi|<r}
e^{i\xi\cdot(x-y)}
e^{t\lambda(\xi)}
\Phi_\xi(\bar x)\Psi_\xi(\bar y)\,d\xi.
\]
Using \eqref{eq:A1-lambda-expansion}, write
\[
\lambda(\xi)
=
i\overline b\cdot\xi-\frac12\,\xi\cdot Q\xi+\rho(\xi),
\qquad
|\rho(\xi)|\le C|\xi|^3
\qquad (|\xi|<r).
\]
Then
\begin{equation}\label{eq:A1-principal-integral}
I_{\mathrm{near}}(t,x,y)
=
\frac{1}{(2\pi)^n}\int_{|\xi|<r}
e^{-i\xi\cdot z}
e^{-\frac t2\xi\cdot Q\xi}
e^{t\rho(\xi)}
\Phi_\xi(\bar x)\Psi_\xi(\bar y)\,d\xi.
\end{equation}
We now make three standard replacements.

\smallskip
\noindent
\emph{(a) Replace \(\Phi_\xi(\bar x)\Psi_\xi(\bar y)\) by \(m(\bar y)\).}
By \eqref{eq:A1-eigenfunction-expansion},
\[
\Phi_\xi(\bar x)\Psi_\xi(\bar y)-m(\bar y)=O(|\xi|)
\qquad \text{uniformly in }x,y.
\]
Moreover,
\[
\left|e^{-\frac t2\xi\cdot Q\xi}e^{t\rho(\xi)}\right|
=
e^{t\Re\lambda(\xi)}
\le e^{-c_0t|\xi|^2}
\qquad (|\xi|<r)
\]
by \eqref{eq:A1-small-xi-damping}. Hence the error caused by replacing \(\Phi_\xi(\bar x)\Psi_\xi(\bar y)\) with \(m(\bar y)\) is bounded by
\[
C\int_{|\xi|<r}|\xi|e^{-c_0t|\xi|^2}\,d\xi
\le Ct^{-(n+1)/2}.
\]

\smallskip
\noindent
\emph{(b) Replace \(e^{t\rho(\xi)}\) by \(1\).}
Since \(|e^w-1|\le |w|e^{|w|}\), we obtain
\[
|e^{t\rho(\xi)}-1|
\le Ct|\xi|^3e^{Ct|\xi|^3}.
\]
After shrinking \(r\) once more if necessary, we have \(Ct|\xi|^3\le \frac{c_0}{2}t|\xi|^2\) for \(|\xi|<r\), and therefore
\[
|e^{t\rho(\xi)}-1|
\le Ct|\xi|^3e^{\frac{c_0}{2}t|\xi|^2}.
\]
Thus the corresponding error is bounded by
\[
C\int_{|\xi|<r}t|\xi|^3e^{-\frac{c_0}{2}t|\xi|^2}\,d\xi
\le Ct^{-(n+1)/2}.
\]

\smallskip
\noindent
\emph{(c) Extend the integral from \(|\xi|<r\) to all of \(\R^n\).}
Since \(Q\) is positive definite, there exists \(\mu>0\) such that
\[
\xi\cdot Q\xi\ge \mu|\xi|^2
\qquad \text{for all }\xi\in\R^n.
\]
Hence
\[
\int_{|\xi|\ge r}e^{-\frac t2\xi\cdot Q\xi}\,d\xi
\le Ce^{-\mu r^2t/4},
\]
which is exponentially small.

Combining (a), (b), and (c), we obtain
\begin{equation}\label{eq:A1-reduced-gaussian}
p(t,x,y)
=
\frac{m(\bar y)}{(2\pi)^n}\int_{\R^n}
e^{-i\xi\cdot z}e^{-\frac t2\xi\cdot Q\xi}\,d\xi
+O(t^{-(n+1)/2})+O(e^{-\eta t})
\end{equation}
uniformly in \(x,y\in\R^n\).

\medskip
\noindent
{\bf Step 5: Evaluation of the Gaussian Fourier integral.}
The standard Gaussian Fourier inversion formula gives
\[
\frac{1}{(2\pi)^n}\int_{\R^n}
e^{-i\xi\cdot z}e^{-\frac t2\xi\cdot Q\xi}\,d\xi
=
\frac{1}{(2\pi t)^{n/2}\sqrt{\det Q}}
\exp\!\left(
-\frac{1}{2t}z\cdot Q^{-1}z
\right).
\]
Substituting this into \eqref{eq:A1-reduced-gaussian} and using the periodicity \(m(\bar y)=m(y)\), we obtain the asserted asymptotic for all sufficiently large \(t\).

It remains to extend the estimate to every \(t\ge 1\). Fix \(t_0\ge 1\) so that the preceding argument is valid for \(t\ge t_0\). For \(t\in[1,t_0]\), formula \eqref{eq:A1-bloch-kernel} and compactness of \([1,t_0]\times B\times\T^n\times\T^n\) imply that
\[
\sup_{1\le t\le t_0}\sup_{x,y\in\R^n}|p(t,x,y)|<\infty.
\]
The Gaussian main term is also uniformly bounded on \([1,t_0]\times\R^n\times\R^n\). Enlarging the constant \(C\) if necessary therefore yields the desired estimate for all \(t\ge 1\). The proof is complete.
\end{proof}

\subsection{Two consequences}

The proof of Theorem~\ref{thm:Large-time Gaussian asymptotics} is stable under compact perturbations of the periodic drift. We record the resulting uniform version, which is the form needed in our application to the Doob-transformed diffusion.

\begin{theorem}[Uniform compact-family version]\label{thm:A1-uniform-family}
Let \(\mathcal P\) be a compact parameter space, and assume that
\[
\alpha\longmapsto b^{(\alpha)}
\]
is continuous from \(\mathcal P\) into \(\Lip(\T^n,\R^n)\). For each \(\alpha\in\mathcal P\), let
\[
A^{(\alpha)}=\frac12\Delta+b^{(\alpha)}(x)\cdot D
\]
and denote by \(p_\alpha(t,x,y)\) its heat kernel on \(\R^n\). Let \(m_\alpha\), \(\overline b_\alpha\), \(\chi_{\alpha,j}\), and \(Q_\alpha\) be the invariant density, effective drift, correctors, and effective diffusion matrix associated with \(A^{(\alpha)}\). Then there exist constants \(C_{\mathcal P},\eta_{\mathcal P}>0\) such that for every \(\alpha\in\mathcal P\), every \(t\ge 1\), and every \(x,y\in\R^n\),

\begin{eqnarray}
    p_\alpha(t,x,y)
&=&
m_\alpha(y)\frac{1}{(2\pi t)^{n/2}\sqrt{\det Q_\alpha}}
\exp\!\left(
-\frac{1}{2t}(y-x-\overline b_\alpha t)\cdot Q_\alpha^{-1}(y-x-\overline b_\alpha t)
\right) \nonumber\\
&+&R_\alpha(t,x,y), \nonumber
\end{eqnarray}

with
\[
|R_\alpha(t,x,y)|
\le
C_{\mathcal P}\,t^{-(n+1)/2}+C_{\mathcal P}e^{-\eta_{\mathcal P}t}.
\]
\end{theorem}

\begin{proof}
We now verify that the five-step proof above can be carried out uniformly on compact parameter sets.

\smallskip
\noindent
\emph{Uniform elliptic bounds and continuous dependence.}
Since \(\mathcal P\) is compact and \(\alpha\mapsto b^{(\alpha)}\) is continuous in \(\Lip\), there exists \(M>0\) such that
\[
\sup_{\alpha\in\mathcal P}\|b^{(\alpha)}\|_{\Lip(\T^n)}\le M.
\]
Hence, all elliptic and parabolic estimates below are uniform in \(\alpha\). Standard Schauder theory, together with uniqueness and the Fredholm alternative on the torus, implies that
\[
\alpha\mapsto m_\alpha,\qquad
\alpha\mapsto \chi_{\alpha,j}
\]
are continuous from \(\mathcal P\) into \(C^{2,\beta}(\T^n)\) for every \(\beta\in(0,1)\). In particular,
\[
\alpha\mapsto \overline b_\alpha,
\qquad
\alpha\mapsto Q_\alpha
\]
are continuous on \(\mathcal P\), and \(\inf_{\alpha\in\mathcal P}\det Q_\alpha>0\).

\smallskip
\noindent
\emph{Uniform Step 1.}
For each \(\alpha\in\mathcal P\) and \(\xi\in B\), define
\[
A_{\alpha,\xi}
:=
e^{-i\xi\cdot x}A^{(\alpha)}e^{i\xi\cdot x}.
\]
The map \((\alpha,\xi)\mapsto A_{\alpha,\xi}\) is continuous from \(\mathcal P\times B\) into \(\mathcal L(H^2(\T^n),L^2(\T^n))\). By parabolic regularity, the kernels \(K_{\alpha,\xi}(t,\cdot,\cdot)\) of \(e^{tA_{\alpha,\xi}}\) are jointly continuous, and for each fixed \(t>0\),
\[
(\alpha,\xi,x,y)\longmapsto K_{\alpha,\xi}(t,x,y)
\]
is continuous on the compact set \(\mathcal P\times B\times\T^n\times\T^n\). In particular,
\[
\sup_{\alpha\in\mathcal P}\sup_{\xi\in B}\sup_{x,y\in\T^n}|K_{\alpha,\xi}(1,x,y)|<\infty.
\]

\smallskip
\noindent
\emph{Uniform Steps 2 and 3.}
For each \(\alpha\in\mathcal P\), the proof of Theorem~\ref{thm:Large-time Gaussian asymptotics} gives a simple isolated eigenvalue \(0\) for \(A^{(\alpha)}\), an analytic principal branch \(\lambda_\alpha(\xi)\), and a positive-definite matrix \(Q_\alpha\). Because \((\alpha,\xi)\mapsto A_{\alpha,\xi}\) is continuous and \(\mathcal P\) is compact, the contour argument in Step~3 may be chosen uniformly in \(\alpha\): there exist \(r\in(0,1)\), \(\eta>0\), and \(C\ge 1\) such that, for every \(\alpha\in\mathcal P\), we have

\begin{itemize}
\item[1.] the branch \(\lambda_\alpha(\xi)\) is analytic for \(|\xi|<r\), and
\[
\lambda_\alpha(\xi)
=
i\overline b_\alpha\cdot\xi-\frac12\,\xi\cdot Q_\alpha\xi+\rho_\alpha(\xi),
\qquad
|\rho_\alpha(\xi)|\le C|\xi|^3
\qquad (|\xi|<r);
\]

\item[2.] the associated right and left eigenfunctions satisfy
\[
\Phi_{\alpha,\xi}=1+i\sum_{j=1}^n\xi_j\chi_{\alpha,j}+O(|\xi|^2),
\qquad
\Psi_{\alpha,\xi}=m_\alpha+O(|\xi|),
\]
uniformly in \((\alpha,\xi)\in\mathcal P\times B_r\);

\item[3.] the remainder estimates in Step~3 hold with constants independent of \(\alpha\):
\[
K_{\alpha,\xi}(t,x,y)
=
e^{t\lambda_\alpha(\xi)}\Phi_{\alpha,\xi}(x)\Psi_{\alpha,\xi}(y)+R_{\alpha,\xi}(t,x,y),
\qquad
|R_{\alpha,\xi}(t,x,y)|\le Ce^{-\eta t}
\]
for \(|\xi|<r\), \(t\ge 1\), and
\[
|K_{\alpha,\xi}(t,x,y)|\le Ce^{-\eta t}
\]
for \(\xi\in B\setminus B_r\), \(t\ge 1\).
\end{itemize}

To justify the uniform far-\(\xi\) spectral gap, one may argue by contradiction exactly as in the proof of \eqref{eq:A1-no-zero-away}: if such a gap failed, there would exist sequences \(\alpha_\ell\in\mathcal P\), \(\xi_\ell\in B\setminus B_r\), and normalized eigenfunctions \(\varphi_\ell\) of \(A_{\alpha_\ell,\xi_\ell}\) with eigenvalues whose real parts converge to \(0\). Compactness of \(\mathcal P\times(B\setminus B_r)\) and elliptic compactness would then yield a nontrivial limit eigenfunction for some \(A_{\alpha,\xi}\) with \(\xi\neq 0\) and spectral bound \(0\), contradicting the single-parameter argument above. The uniform isolation of the principal eigenvalue near \(\xi=0\) is obtained in the same way.

\smallskip
\noindent
\emph{Uniform Steps 4 and 5.}
With the uniform constants established above, the three replacements in Step~4 and the Gaussian inversion in Step~5 are unchanged. Every estimate is uniform in \(\alpha\), because the constants controlling \(Q_\alpha\), \(\det Q_\alpha\), the remainder \(\rho_\alpha\), and the fiber-kernel decomposition are uniform on \(\mathcal P\). Therefore the conclusion of Theorem~\ref{thm:Large-time Gaussian asymptotics} holds with constants independent of \(\alpha\).
\end{proof}

We now return to the original Schr\"odinger kernel appearing in Section~\ref{sec:ballistic}. Combining the Doob \(h\)-transform with Theorem~\ref{thm:A1-uniform-family} yields a refined ballistic asymptotic with an explicit periodic amplitude.

\begin{proposition}[Sharp ballistic asymptotics for the original Schr\"odinger kernel]\label{prop:A1-sharp-ballistic}
Keep the notation of Section~\ref{sec:ballistic}. Let \(V_0\subset\R^n\) and \(R_0\subset\R^n\) be compact. Then there exist constants \(C,\eta>0\) such that, for every \(q\in V_0\), every \(r\in R_0\), and every \(t\ge 1\),
\begin{align*}
&\quad K(t,0,-qt+r)\\
&=
e^{-t\overline L(q)}
\exp\bigl(p(q)\cdot r+v_{p(q)}(-qt+r)-v_{p(q)}(0)\bigr)\\
&\times
\Bigg[
\pi_{p(q)}(-qt+r)\,
\frac{1}{(2\pi t)^{n/2}\sqrt{\det Q_{p(q)}}}
\exp\!\left(-\frac{1}{2t}r\cdot Q_{p(q)}^{-1}r\right)
+O\!\left(t^{-(n+1)/2}\right)+O(e^{-\eta t})
\Bigg],
\end{align*}
where \(p(q)\) is the unique vector satisfying \(D\overline H(p(q))=q\), and \(Q_{p(q)}\) is the effective diffusion matrix of the periodic diffusion generated by
\[
\mathcal G_{p(q)}
=
\frac12\Delta+(-p(q)-Dv_{p(q)}(x))\cdot D.
\]
Moreover, in the present mechanical setting, one has \(Q_p=D^2\overline H(p)=(D^2 \overline L(q))^{-1}\); this is the standard identification between the effective diffusion matrix of the Doob-transformed process and the Hessian of the effective Hamiltonian, together with the Legendre duality.

In particular, for \(r=0\),
\[
K(t,0,-qt)
=
e^{-t\overline L(q)}\,a_{p(q)}(-qt)\,t^{-n/2}
+
O\!\left(e^{-t\overline L(q)}t^{-(n+1)/2}\right),
\]
where the periodic amplitude is
\[
a_p(y)
:=
e^{v_p(y)-v_p(0)}\pi_p(y)\frac{1}{(2\pi)^{n/2}\sqrt{\det Q_p}}.
\]
The error term is uniform for \(q\in V_0\).
\end{proposition}

\begin{remark}
Proposition~\ref{prop:A1-sharp-ballistic} is not a direct restatement of \cite[Theorem~2.2]{HKN20}. The latter gives a powerful large-deviation asymptotic for the branching-diffusion kernel \(u(t,x,y)\) in periodic media, expressed in terms of the rate function \(\Phi\) and the associated tilted principal eigenfunctions. Our formulation is tailored to the present homogenization problem: it is stated directly for the original Schr\"odinger kernel \(K\) along the ballistic rays \(y=-qt+r\), with the prefactor written in the variables naturally associated with the cell problem, namely \(p(q)\), \(v_{p(q)}\), \(\pi_{p(q)}\), and \(Q_{p(q)}\), uniformly for \(q\in V_0\) and \(r\in R_0\). We also note that our proof is organized somewhat differently, through the uniform compact-family Gaussian asymptotics of Theorem~\ref{thm:A1-uniform-family} together with the Doob \(h\)-transform, so as to fit the proof of Theorem~\ref{theo:main}.
\end{remark}

\begin{proof}
By the Doob \(h\)-transform identity proved in Section~\ref{sec:ballistic},
\[
K(t,0,y)
=
e^{t\overline H(p)}\frac{h_p(0)}{h_p(y)}\,\widetilde p_p(t,0,y),
\qquad
h_p(x)=e^{-p\cdot x-v_p(x)}.
\]
Taking \(y=-qt+r\) and using \(q=D\overline H(p)\), we obtain
\[
K(t,0,-qt+r)
=
e^{-t\overline L(q)}
e^{p\cdot r+v_p(-qt+r)-v_p(0)}
\widetilde p_p(t,0,-qt+r).
\]

Now let
\[
P:=\{p\in\R^n:\ D\overline H(p)\in V_0\}.
\]
Since \(\overline H\) is strictly convex and superlinear in the quadratic setting, \(D\overline H:\R^n\to\R^n\) is a homeomorphism onto \(\R^n\); see \cite[Chapter~2]{Tran}. Hence \(P\) is compact, and \(q\mapsto p(q)\) is continuous on \(V_0\).

For each \(p\in P\), consider the periodic diffusion with generator
\[
\mathcal G_p
=
\frac12\Delta+(-p-Dv_p(x))\cdot D.
\]
Standard elliptic regularity for the cell problem \eqref{eq:cell-p} implies that \(p\mapsto v_p\) is continuous from \(P\) into \(C^{2,\beta}(\T^n)\) for every \(\beta\in(0,1)\). Therefore
\[
p\longmapsto b^{(p)}(x):=-p-Dv_p(x)
\]
is continuous from \(P\) into \(\Lip(\T^n,\R^n)\). Theorem~\ref{thm:A1-uniform-family} thus applies to the family \(\{\mathcal G_p\}_{p\in P}\).

For this family, the effective drift equals
\[
\overline b_p
=
\int_{\T^n}(-p-Dv_p)\pi_p\,dx
=
-D\overline H(p)
=
-q.
\]
Thus, with \(x=0\) and \(y=-qt+r\),
\[
y-x-\overline b_pt=r.
\]
Applying Theorem~\ref{thm:A1-uniform-family} gives
\begin{align*}
    &\quad\widetilde p_p(t,0,-qt+r)\\
&=
\pi_p(-qt+r)\frac{1}{(2\pi t)^{n/2}\sqrt{\det Q_p}}
\exp\!\left(-\frac{1}{2t}r\cdot Q_p^{-1}r\right)
+O\!\left(t^{-(n+1)/2}\right)+O(e^{-\eta t}),
\end{align*}

uniformly for \(p\in P\), \(r\in R_0\), and \(t\ge 1\). Substituting this into the Doob-transform identity proves the first assertion. The case \(r=0\) follows immediately, since the exponentially small term may be absorbed into \(O(t^{-(n+1)/2})\) for \(t\ge 1\).
\end{proof}

\begin{remark}
We recall that \(\pi_p=C_p r_p r_{-p}\), where the constant \(C_p>0\) is chosen so that \(\int_{\T^n}\pi_p\,dx=1\). Hence
    \begin{align*}
        &e^{-t\overline L(q)}
\exp\bigl(p\cdot r+v_p(-qt+r)-v_p(0)\bigr) \pi_p(-qt+r)
\exp\!\left(-\frac{1}{2t}r\cdot Q_p^{-1}r\right)\\
&=
C_p\exp\!\left(-t \left(\overline L(q)-D\overline L(q)\cdot \frac{r}{t}+\frac{1}{2}\frac{r}{t}\cdot D^2 \overline L(q)\frac{r}{t}\right) \right) \exp\!\left(-v_p(0)-v_{-p}(-qt+r)\right).
    \end{align*}
Since \(\overline L\in C^2(\R^n)\), Taylor's theorem gives
\[
\overline L\left(q-\frac{r}{t}\right)
=
\overline L(q)-D\overline L(q)\cdot \frac{r}{t}
+\frac{1}{2}\frac{r}{t}\cdot D^2\overline L(q)\frac{r}{t}
+O(t^{-3}),
\]
uniformly for \(q\in V_0\) and \(r\in R_0\). Therefore Proposition~\ref{prop:A1-sharp-ballistic} also yields
\begin{multline*}
K(t,0,-qt+r)=\\
C_p\exp\!\left(-t \overline L\left(q-\frac{r}{t}\right) \right) \exp\!\left(-v_p(0)-v_{-p}(-qt+r)\right)\Bigg[
\frac{\sqrt{\det D^2 \overline L(q)}}{(2\pi t)^{n/2}}
+\ O\!\left(t^{-(n+1)/2}\right)
\Bigg],
\end{multline*}
where the extra \(O(t^{-2})\) produced by Taylor's theorem has been absorbed into the displayed remainder.
\end{remark}



\end{document}